\documentclass[a4paper]{amsart}

\usepackage{a4wide,amssymb,amsthm,amsmath,bbm,epic,eepic,graphics,amsbsy,url,graphicx,verbatim,xypic}
\usepackage{tikz}
\usepackage{tikz-cd}
\tikzset{
  symbol/.style={
    draw=none,
    every to/.append style={
      edge node={node [sloped, allow upside down, auto=false]{$#1$}}}
  }
}
\usepackage{lscape}
\usepackage{afterpage}
\usepackage{capt-of}

\usepackage[hyperindex
,bookmarks]
{hyperref}
\usepackage{enumitem}

\allowdisplaybreaks

\def\C{\mathbb{C}}

\def\Z{\mathbb{Z}}
\def\Q{\mathbb{Q}}
\def\GL{\mathrm{GL}}
\def\eps{\epsilon} 
\def\Hom{\mathrm{Hom}}

\def\kk{\mathbf{k}}
\def\into{\hookrightarrow}

\def\dlog{\operatorname{dlog}}

\def\Br{\mathrm{Br}} 
\def\Pure{\mathrm{P}} 
\def\W{\mathrm{S}} 
\def\Grp{\Gamma} 
\def\Ar{\mathcal{A}} 
\def\Stab{\operatorname{\mathrm{Stab}}}
\def\StabWn{\operatorname{\mathrm{Stab}_{\W_n}\!}} 
\def\sgn{\eps} 
\def\qsgn{\Q_\sgn} 
\def\R{\mathrm{R}}
\def\Graphs{\mathcal{D}} 
\def\graph{D} 
\def\full{K} 

\def\bune#1{\phantom{_#1}\bullet^{#1}}
\def\buse#1{\phantom{^#1}\bullet_{#1}}
\def\buno#1{^{#1}\bullet\phantom{_#1}}
\def\buso#1{_{#1}\bullet\phantom{^#1}}

\title{
Cohomology of quasi-abelianized braid groups
}
\author{Filippo Callegaro \& Ivan Marin}
\date{\today}
\address[F. Callegaro]{Dipartimento di Matematica, University of Pisa, Italy}
\email{callegaro@dm.unipi.it}
\address[I. Marin]{LAMFA UMR CNRS 7352, Universit\'e de Picardie Jules Verne, Amiens, France}
\email{ivan.marin@u-picardie.fr}

\theoremstyle{plain}
\newtheorem{theorem}{Theorem}

\newtheorem{lemma}{Lemma}
\newtheorem{proposition}{Proposition}
\newtheorem{corollary}{Corollary}

\theoremstyle{definition}
\newtheorem{definition}{Definition}
\newtheorem{remark}{Remark}

\setlist[enumerate]{label=\alph*)}
\begin{document}
\begin{abstract}
	We investigate the rational cohomology of the quotient of (generalized) braid groups
	by the commutator subgroup of the pure braid groups. We provide a combinatorial
	description of it using isomorphism classes of certain families of graphs. We establish Poincar\'e dualities for them and prove a stabilization property for the infinite series of reflection groups.
\end{abstract}
\maketitle

\tableofcontents

\section{Introduction}

Let $\Br_n$ be the braid group on $n$ strands, $\Pure_n \subset \Br_n$ the pure braid group on $n$ strands, $\W_n = \Br_n/\Pure_n$ the symmetric group on $n$ letters. The quasi-abelianization of the braid group is the quotient $\Grp_n$ of the braid group by the
commutator subgroup of the \emph{pure} braid group. More generally, when $B$ is an Artin group of finite Coxeter type or, even more generally,
the generalized braid group associated to a (finite) complex reflection group $W$, its quasi-abelianization is its quotient by the commutator
subgroup of the associated pure braid group.

These groups, whose first implicit introduction can arguably be traced back to the same paper where Tits (also implicitely) first defined
the now called Artin groups (see \cite{TITS}), have gain interest in the past decades. They play a role in a number of works where
linear actions of braids are involved, as some of the linear representations of the braid groups appearing in nature factor
through this group, the pure braids being sent to diagonal (and therefore commuting) matrices (see \cite{TONGYANGMA,SYSOEVA,THMARIN} among others). 
More recently, the group $\Gamma_n$ has been shown to have a nice presentation and a nice categorical interpretation in \cite{PANAITESAIC}.

It was observed in \cite{GGO} that each $\Gamma_n$ is a crystallographic group, and also proved there that it has not any element of
order $2$, which has for the consequence that torsion-free crystallographic groups (also known as Bieberbach groups) can be constructed
easily from it. Subsequently, it was proved in \cite{CRYSTMARIN} that these properties are actually true for any group $\Gamma$ similarly associated
to an arbitrary complex reflection group, and in \cite{BECKMAR} that, however, many finite groups of odd order can be embedded
in $\Gamma$ (and actually any of them inside $\Gamma_n$ for some $n$).

In this paper we explore the rational cohomology of these groups, and show that it admits a nice combinatorial
description in terms of isomorphism classes of certain families of graphs. Moreover,
we exhibit a few remarkable phenomena. Among them, for $\Gamma_n$ and more generally for $\Gamma$
when $W$ is a complex reflectiong group of type $G(de,e,n)$, we show that the rational cohomology stabilizes,
as in the case of the usual braid group (see \cite{ARNOLD}) and other related groups (see e.g. \cite{CALMAR}, \cite{CALSALV}, \cite{RWW}, \cite{W}).
Also, we show that there is a Poincar\'e duality that
can appear under various disguise, and which also admits a simple combinatorial interpretation.

The first part of the paper is devoted to a detailed study of the case of the usual braid group.
We then consider the general case in Section \ref{sect:generalbraidgroups}, establish general properties there, and
then provide a combinatorial description in the case of the general series $G(de,e,n)$ of complex
reflection groups. This enables us to prove the general stabilization result in Section \ref{sect:stabilization}.

\section{Braid group case}

\subsection{General properties}

Let $\Br_n$ be the braid group on $n$ strands, $\Pure_n \subset \Br_n$ the pure braid group on $n$ strands, $\W_n = \Br_n/\Pure_n$ the symmetric group on $n$ letters.

\begin{definition}
We consider \emph{quasi-abelianized braid group}, that is the quotient $$\Grp_n := \Br_n/[\Pure_n, \Pure_n].$$ and we write $\Z\Ar_n$ for the abelianization of the pure braid group, that is the quotient $\Pure_n/[\Pure_n, \Pure_n] \simeq \Z^{\binom{n}{2}}$. 
\end{definition}
\begin{remark}
Since $\Pure_n$ is the fundamental group
of the complement in $\C^n$ of the hyperplanes defined by the equations $z_i = z_j$, for $i < j$,
 this group $\Z\Ar_n$ can be identified with
its first homology groups. Therefore, it is freely generated by one element  
per hyperplane $z_i = z_j$ (see e.g. \cite{ORLIKTERAO}), and the symbols $\omega_{ij}$, for $1 \leq i < j \leq n$ denote the elements of the dual basis. Inside the corresponding de Rham first cohomology group
it corresponds to (the class of) the 1-form $\dlog (z_i - z_j)$.
We order these $\omega_{ij}$ lexicographically.
\end{remark}

\begin{proposition}\label{prop:main}
For any $\Q \W_n$-module $M$ the following isomorphism holds:
$$
H^q(\Grp_n; M) = H^q(\Z\Ar_n; M)^{\W_n}.
$$
\end{proposition}
\begin{proof}
The group $\Grp_n$ fits in the short exact sequence
$$
1 \to \Z\Ar_n \to \Grp_n \to \W_n \to 1.
$$
Hence the spectral sequence
$$
E_2^{p,q} = H^p(\W_n; H^q(\Z\Ar_n; M))
$$
for $p+q = r$ converges to 
the group $H^r(\Grp_n; M)$.
The group $W_n$ is finite and the module $M$ is divisible: using classical results in group cohomology (see for example \cite[Cor.~II,~5.4]{ADEMMILGRAM}) this implies that $H^p(\W_n; H^q(\Z\Ar_n; M))=0$ for $p>0$. 
Finally, since $H^0(\W_n; H^q(\Z\Ar_n; M)) = H^q(\Z\Ar_n; M)^{\W_n}$ we obtain the claim.
\end{proof}
\begin{remark}
	The cohomology ring $H^{\bullet}(\Z\Ar_n; \Q)$ is the exterior algebra freely generated by $$H^1(\Z\Ar_n;\Q) = \Hom(\Z\Ar_n;\Q) \simeq \Q\Ar_n \simeq  \Q^{\binom{n}{2}}$$ in degree $1$. The group $\W_n$ acts on the algebra $H^{\bullet}(\Z\Ar_n; \Q)$ as an automorphism group mapping the degree 1 elements as follows:
a permutation $\sigma \in \W_n$ maps the generator $\omega_{ij}$ to
\begin{equation}\label{eq:sigma_action}
\sigma(\omega_{ij}) = 
\left\{
\begin{array}{ll}
\omega_{\sigma(i),\sigma(j)}  & \mbox{ if } \sigma(i) <\sigma(j) \\
\omega_{\sigma(j),\sigma(i)} & \mbox{ otherwise.}
\end{array}
\right.
\end{equation}
\end{remark}
\begin{definition}
	We write $\full_n$ for the full graph with vertices $\{1, \ldots, n\}.$
\end{definition}
\begin{definition} \label{def:mu_delta}
For a graph $\Delta \subset \full_n$, let $e_1=(i_1, j_1), \ldots, e_k=(i_k, j_k)$ be the list of edges of $\Delta$ ordered lexicographically.
We define 
\begin{equation}\label{eq:def_mu_Delta}
	\mu_\Delta:=\omega_{i_1j_1} \cdots \omega_{i_kj_k} \in H^k(\Z\Ar_n; \Q).
\end{equation}
\end{definition}

\subsection{Combinatorial description of $H^{\bullet}(\Gamma_n;\Q)$}
\label{sect:combHGammanQ}
We consider the natural action of $\W_n$ on the subgraphs of $\full_n$.
Given a graph $\Delta \subset \full_n$ we can decompose the group $\W_n$ as a union of the cosets $C_1(\Delta), \ldots, C_s(\Delta)$ of $\StabWn(\Delta)$.

\begin{definition}We set the following notations:
\begin{enumerate}[label=\alph*)]
\item We denote by $\mathcal{G}$ the set of isomorphism classes of finite graphs. 

\item For a given graph $\Delta$ we write $[\Delta]$ for its isomorphism class in $\mathcal{G}$.

\item We say that $\Delta$ is an \emph{invariant graph} if all the automorphisms of the graph induce even permutations on the set of edges.

\item Let $\Graphs \subset \mathcal{G}$ be the set of (isomorphism classes of) invariant graphs.

\item Define $\Graphs_n$ as the subset of $\Graphs$ of graphs with exactly $n$ vertices (allowing isolated vertices). 
\end{enumerate}
\end{definition}

For every isomorphism class $\graph \in \mathcal{G}$ of a graph with $n$ vertices and without isolated vertices we choose once and for all a representative $\Delta_\graph \subset \full_n$ of $\graph$.
This is equivalent to say that we fix 
a total ordering of the vertices of $\graph$. 
For an arbitrary isomorphism class $\graph$, we consider the isomorphism class of the maximal subgraph $\graph_0$ without isolated points, and choose as a 
representative $\Delta_\graph$ of $\graph$ the subgraph of $\Pi_{|V\graph|}$ that coincides with the representative of  $\graph_0$
on $[1, |V\graph_0|]$.

We write $(i,j)$ with $i<j$ for the edge with set of vertices $\{i,j\}$.

\begin{theorem}\label{thm:generators}
The set of cohomology classes
\begin{equation}\label{eq:alpha_graph}
	\alpha_\graph:= \frac{1}{\StabWn(\Delta_\graph)}\sum_{\sigma \in \W_n} \sigma(\mu_{\Delta_\graph}) \in H^{|ED|}(\Z\Ar_n; \Q)
\end{equation}
for $\graph \in \Graphs_n$  is a basis of $H^\bullet(\Z\Ar_n; \Q)^{\W_n}$ over $\Q$.
\end{theorem}
\begin{proof}
Note that if $\graph \in \mathcal{D}$ then the group $\StabWn(\Delta_\graph)$ acts trivially on $\mu_{\Delta_\graph}$. Hence the average $\pi_{\W_n}(\mu_{\Delta_\graph}) = \dfrac{1}{|\W_n|}\sum_{\sigma \in \W_n} \sigma(\mu_{\Delta_\graph}) = \dfrac{|\StabWn(\Delta_\graph)|}{n!}  \alpha_{\graph}$ is a non-zero $\W_n$-invariant. If $\graph \notin \mathcal{D}$ then there is a permutation $\sigma \in \W_n$ such that $\sigma(\mu_{\Delta_\graph}) = -\mu_{\Delta_\graph}$ and hence the average of $\mu_{\Delta_\graph}$ is zero.
The Theorem follows since $\pi_{\W_n}$ is a projection from $H^\bullet(\Z\Ar_n; \Q)$ to $H^\bullet(\Z\Ar_n; \Q)^{\W_n}$ and the elements $\alpha_\graph$ for $\graph \in \mathcal{D}$ are linearly independent.
\end{proof}

We can map a graph of $\mathcal{D}_n$ to graph of $\mathcal{D}_{n+1}$
by adding an isolated vertex. With this identification we can write  $\mathcal{D}_n \subset \mathcal{D}_{n+1}$ and $\mathcal{D} = \cup_n \mathcal{D}_n$.

\begin{proposition}\label{prop:stab_untwisted}
The natural morphism $\Gamma_n \hookrightarrow \Gamma_{n+1}$ induces a map $H^{\bullet}(\Gamma_{n+1};\Q) \to H^{\bullet}(\Gamma_{n};\Q)$.
The map send the class $\alpha_{\graph} \in H^{\bullet}(\Gamma_{n+1};\Q)$  to the corresponding class $\alpha_{\graph} \in H^{\bullet}(\Gamma_{n};\Q)$ when  $\graph \in \Graphs_{n}$, otherwise to the zero class.

Since every graph with $r$ edges without isolated vertices has at most $2r$ vertices
we get that $H^r(\Gamma_n;\Q)$ stabilizes for $r \geq 2n$. 
\end{proposition}
\begin{proof}
	It is enough to check that if $\graph \notin \Graphs_{n}$ then all the summand of $\alpha_{\graph}$ contains a factor $\omega_{i,n+1}$ and hence $\alpha_{\graph}$ restricts to zero in $H^{\bullet}(\Gamma_{n};\Q)$, while if $\graph \in \Graphs_{n}$ then $\alpha_{\graph} \in H^{\bullet}(\Gamma_{n+1};\Q)$ restricts to $\alpha_{\graph} \in H^{\bullet}(\Gamma_{n};\Q)$.
\end{proof}
As a consequence of the Proposition above, the series $\sum \dim H^r(\Gamma_n;\Q)t^r$ has for limit when $n \to \infty$ the Poincar\'e series of the colimit $\Gamma_{\infty}$ of the embeddings $\Gamma_n \into \Gamma_{n+1}$, 
whose $r$-th Betti number is then
equal to the number of invariant graphs with $r$ edges, without isolated vertices.

We can actually say something more about the stable Poincaré series. The proof of following lemma is straightforward.
\begin{lemma}
A graph $\Delta$ is invariant if and only if the following conditions hold:
\begin{enumerate}[label=\alph*)]
	\item the connected components of $\Delta$ are invariant graphs;
	\item any  two distinct connected components of $\Delta$ with odd edges are not isomorphic.
\end{enumerate}
\end{lemma}
\begin{corollary}
The cohomology group $H^\bullet(\Gamma_n;\Q)$ is isomorphic to the tensor product of a polinomial algebra generated by connected invariant graphs with even edges and an exterior algebra generated by connected invariant graphs with odd edges. The degree of each generator equals the number of edges of the graph.

The Poincar\'e series of $\Gamma_{\infty}$ decomposes as follows:
\begin{equation}\label{eq:poinc}\sum \dim H^r(\Gamma_\infty;\Q)t^r= \left(\prod_{\substack{\graph \in \Graphs\\ \graph \mbox{ connected}\\|E\graph| \mbox{ odd}}} (1+t^{|E\graph|})\right)\cdot \left(\prod_{\substack{\graph \in \Graphs\\ \graph \mbox{ connected}\\|E\graph| \mbox{ even}}} (1+t^{|E\graph|})\right)^{-1}.\end{equation}
\end{corollary}
This description does not give an account of the actual multiplicative structure of the cohomology ring $H^\bullet(\Gamma_n;\Q)$, that will be investigated in Theorem \ref{theo:product}, but only of the multiplicative structure of the graded algebra associated to the filtration defined by the number of connected components of a graph.

\subsection{Combinatorial description of the twisted cohomology group $H^{\bullet}(\Gamma_n;\qsgn)$}
\begin{definition}
	We write $\sgn(\sigma)$ for the sign of a permutation $\sigma \in \W_n$ and let $\qsgn$ be the 1-dimensional sign rational representation of $\W_n$ extended to $\Gamma_n$.
\end{definition}
The cohomology group $H^{\bullet}(\Z \Ar_n; \qsgn)$
is naturally identified, as an $\W_n$-representation, with the tensor product of the exterior algebra on
the $\Q \W_n$-module $\Q \Ar_n$
and $\qsgn$. That is
\begin{equation}\label{eq:skew_cohom}
H^{r}(\Z \Ar_n; \qsgn) \simeq \qsgn
\otimes H^{r}(\Z \Ar_n; \Q).
\end{equation}

Here  we provide a combinatorial description of a basis of the subspace of $\W_n$-invariant elements.

\begin{remark}
We note that since $\qsgn$ is a trivial $\Z \Ar_n$-representation, the cohomology groups $H^\bullet(\Z \Ar_n; \Q)$ and $H^\bullet(\Z \Ar_n; \qsgn)$ are isomorphic as graded vector spaces.
These groups differs only as $\W_n$-representation and as $\Gamma_n$-representations.
\end{remark}

Therefore we can extend the notation of Definition \ref{def:mu_delta} to the group $H^\bullet(\Z \Ar_n; \qsgn)$: for a graph $\Delta \subset \full_n$ we allow to write $\mu_\Delta \in H^\bullet(\Z \Ar_n; \qsgn)$ for the product of the ordered set of generator $\omega_{ij}$ corresponding to the edges of $\Delta$.

\begin{remark}
The group $H^\bullet(\Z \Ar_n; \Q)$ has a natural ring structure given by cup product, and this structure is clearly compatible with the structure of trivial $\W_n$-module. On the other side, since $\qsgn \otimes \qsgn = \Q$ is the trivial $\W_n$ representation, the ring structure of $H^\bullet(\Z \Ar_n; \qsgn)$ given by cup product is not compatible with the structure of $\W_n$-module. Hence it will be natural to consider a different product structure (see Section  \ref{sect:multiplicative_struct}) that maps
$$
H^i(\Z \Ar_n; \qsgn) \otimes H^j(\Z \Ar_n; \qsgn) \to H^{i+j}(\Z \Ar_n; \Q).
$$ 
\end{remark}

\begin{definition}
Let $\Delta$ be a graph with $n$ vertices. For any automorphism $\sigma$ of $\Delta$ we identify $\sigma$ with the corresponding permutation of the set of vertices of $\Delta$ and we write $\sigma_{E\Delta}$ for the induced permutation of the set of edges of $\Delta$.

More generally, let $\Delta, \Delta' \subset \full_n$ be two isomorphic graphs. Let $e_1, \ldots, e_k$ be the ordered list of edges of $\Delta$ and $e_1', \ldots, e_k'$ the ordered list of edges of $\Delta'$. Let $\sigma \in \W_n$ be a permutation such that $\sigma(\Delta) = \Delta'.$ We write $\sigma_{E\Delta} \in \W_k$ for the permutation induced  by $\sigma$ on the set of the edges, that gives
$\sigma(e_i) = e_{\sigma_{E\Delta}(i)}',$ for $i \in 1, \ldots, k$.
\end{definition}

\begin{definition}
	Given a permutation $\sigma \in \W_n$ that induces an automorphism on the graph $\Delta$ we introduce the following notation for the sign of the induced permutation $\sigma_{E\Delta}$ on the set of edges: $$\sgn_{\Delta}(\sigma) := {\sgn(\sigma_{E\Delta})}.$$
\end{definition}

\begin{definition}
We define $\Graphs^{\sgn}_n$ to be the set of isomorphism classes of graphs with exactly $n$ vertices
such that any automorphism $\sigma$ of the graph satisfies $\sgn(\sigma)
\sgn_\Delta(\sigma) = 1$. We call those graphs \emph{skew-invariant graphs}.
\end{definition}

\begin{remark}
	Note that if $[\Delta] \in \Graphs^{\sgn}_n$ 
	then $\Delta$ has at most 1 isolated point. In fact if there are at least 2 isolated point, there is an automorphism of the graph that fixes all the edges, but permutes the isolated points with an odd permutation. 
\end{remark}

The proof of the following two lemmas is straightforward from the definitions.
\begin{lemma} In $H^{\bullet}(\Z \Ar_n; \Q)$ the following equality holds:
	$$
	\sigma(\mu_\Delta) = {\sgn_\Delta(\sigma)}\mu_{\sigma(\Delta)}.
	$$
\end{lemma}

\begin{lemma} In $H^{\bullet}(\Z \Ar_n; \qsgn)$ the following equality holds:
	$$
	\sigma(\mu_\Delta) = \sgn(\sigma)
	{\sgn_\Delta(\sigma)}\mu_{\sigma(\Delta)}.
	$$
\end{lemma}

\begin{theorem}\label{thm:skewgenerators}
The set of cohomology classes 
\begin{equation}\label{eq:alpha_sgn}
\alpha_\graph^{\sgn}:= \frac{1}{\StabWn(\Delta_\graph)}\sum_{\sigma \in \W_n} 
\sigma(\mu_{\Delta_\graph}) \in H^{|E\graph|}(\Z \Ar_n; \qsgn) 
\end{equation}
for 
$\graph \in \Graphs^{\sgn}_n$
is a basis of $H^{\bullet}(\Grp_n; \qsgn) \subset H^{\bullet}(\Z \Ar_n; \qsgn)$.  
\end{theorem}
The proof of the Theorem 
above is analogous to the one of Theorem \ref{thm:generators} and we omit it. 

We note in particular that when 
$\graph \notin \Graphs_n^\sgn$
then the cohomology class $\alpha_{\graph}^{\sgn}\in H^{|E\graph|}(\Z \Ar_n; \Q)$ defined above is trivial.

In the following Proposition we use the identification $\mathcal{D}_n^\sgn \subset \mathcal{D}_{n+1}^\sgn$ given by adding one isolated vertex. 
\begin{proposition}\label{prop:stable_skew}
The natural morphism $\Gamma_n \hookrightarrow \Gamma_{n+1}$ induces the restriction map $H^{\bullet}(\Gamma_{n+1};\qsgn) \to H^{\bullet}(\Gamma_{n};\qsgn)$. 
%
If $\graph \in \Graphs_{n+1}^\sgn$ is a graph with exactly one isolated vertex, the restriction of $\alpha_{\graph}^{\sgn} \in H^{r}(\Z \Ar_{n+1}; \qsgn)$ is $\alpha_{\graph}^{\sgn} \in H^{r}(\Z \Ar_{n}; \qsgn)$. If $\graph$ is a graph without isolated vertices $\alpha_{\graph}^{\sgn}$ restricts to zero.   In particular the restriction homomorphism $H^{\bullet}(\Gamma_{n+1};\qsgn) \to H^{\bullet}(\Gamma_{n};\qsgn)$ is injective when restricted to the subspace generated by the classes  $\alpha_{\graph}^{\sgn}$ with $\graph$ a graph with exactly one isolated vertex. 
\end{proposition}

The proof is analogue to the proof of Proposition \ref{prop:stab_untwisted} and we omit it.

Since a graph with $r$ edges contains at least $n-2r$ isolated vertices
we deduce from the Proposition above that 
$H^{r}(\Gamma_n;\qsgn) = 0$ as soon as $n \geq 2r+2$. Hence we have the following consequence. 
\begin{corollary}
The stable skew cohomology $H^{\bullet}(\Gamma_{\infty};\qsgn)$ is zero. 
\end{corollary}

\subsection{Multiplicative structure}\label{sect:multiplicative_struct}

We begin describing the multiplicative structure of the cohomology of $\Z\Ar_n$ with untwisted and twisted rational coefficients and how they interplay.
 
\begin{lemma}
The isomorphisms of $\W_n$-representations 
\begin{align}
&\Q \otimes \Q  \simeq  \Q, \label{eq:untwisted}\\
&\qsgn \otimes \qsgn  \simeq  \Q \label{eq:skew}\\
&\Q \otimes \qsgn  \simeq   \qsgn, \label{eq:twisted} 
\end{align}
induces natural bilinear products of $\W_n$-representations
\begin{align*}
&\cup:H^\bullet(\Z  \Ar_n;\Q) \otimes H^\bullet(\Z  \Ar_n;\Q) \to H^\bullet(\Z \Ar_n;\Q),\\
&\cup^s:H^\bullet(\Z \Ar_n;\qsgn) \otimes H^\bullet(\Z \Ar_n;\qsgn) \to H^\bullet(\Z \Ar_n;\Q),\\
&\cup^t:H^\bullet(\Z \Ar_n;\Q) \otimes H^\bullet(\Z \Ar_n;\qsgn) \to H^\bullet(\Z \Ar_n;\qsgn),
\end{align*}
\end{lemma}
\begin{proof}
The first product is the usual cap product for the cohomology ring $H^\bullet(\Z \Ar_n;\Q)$.

All tree products are actually a form of cup product, since they are obtained composing the cross products
\begin{align*}
	&\times:H^\bullet(\Z  \Ar_n;\Q) \otimes H^\bullet(\Z  \Ar_n;\Q) \to H^\bullet(\Z \Ar_n\times \Z \Ar_n;\Q \otimes \Q),\\
	&\times^s:H^\bullet(\Z \Ar_n;\qsgn) \otimes H^\bullet(\Z \Ar_n;\qsgn) \to H^\bullet(\Z \Ar_n \times \Z \Ar_n;\qsgn \otimes \qsgn),\\
	&\times^t:H^\bullet(\Z \Ar_n;\Q) \otimes H^\bullet(\Z \Ar_n;\qsgn) \to H^\bullet(\Z \Ar_n\times \Z \Ar_n;\Q \otimes \qsgn),
\end{align*}
with the diagonal map and the homomorphism induced by the map of coefficients given by equations \eqref{eq:untwisted}, \eqref{eq:skew} and \eqref{eq:twisted} (see for example \cite[ch.~5.6]{spanier}).
\end{proof}
\begin{remark}
We note that the products above can be easily computer in terms of the tensor product and of the usual product of the exterior algebra:
\begin{align*}
	&\cup: \Lambda^\bullet(\Q\Ar_n) \otimes \Lambda^\bullet(\Q\Ar_n) \to \Lambda^\bullet(\Q\Ar_n),\\
	&\cup^s: (\qsgn \otimes \Lambda^\bullet(\Q\Ar_n)) \otimes (\qsgn \otimes \Lambda^\bullet(\Q\Ar_n)) \to (\qsgn \otimes \qsgn) \otimes \Lambda^\bullet(\Q\Ar_n) \simeq \Lambda^\bullet(\Q\Ar_n),\\
	&\cup^t: \Lambda^\bullet(\Q\Ar_n) \otimes (\qsgn \otimes \Lambda^\bullet(\Q\Ar_n))  \to \qsgn  \otimes \Lambda^\bullet(\Q\Ar_n).
\end{align*}
\end{remark}
\begin{definition}
For a given graph $\Delta$, let $S(\Delta, \Delta_1, \Delta_2)$ be the set of pairs $(\Delta', \Delta'')$ such that 
\begin{enumerate}
\item $\Delta', \Delta''$ are subgraph of $\Delta$ 
such that $E\Delta' \sqcup E\Delta'' = E\Delta$;
\item $\Delta'$ is isomorphic to $\Delta_1$ and $\Delta''$ is isomorphic to $\Delta_2$.
\end{enumerate}
\end{definition}
\noindent For $(\Delta', \Delta'') \in S(\Delta, \Delta_1, \Delta_2)$ we order the sets vertices $V\Delta'$ and $V\Delta''$ with the ordering induced by the ordering of $V\Delta$. Let $\sigma_1, \sigma_2 \in \W_n$ such that $\sigma_1(\Delta_1) = \Delta'$ and $\sigma_2(\Delta_2) = \Delta''$. 
\begin{definition}\label{def:coeff_plain}
For $[\Delta],[\Delta_1], [\Delta_2] \in \Graphs_n$ and $(\Delta', \Delta'') \in S(\Delta, \Delta_1, \Delta_2)$ we define the coefficient $\sgn{(\Delta,\Delta', \Delta'')} \in \{+1, -1\}$ according to the following formula:
$$
\mu_\Delta = \sgn{(\Delta,\Delta', \Delta'')} \sgn_{\Delta_1}(\sigma_1) \sgn_{\Delta_2}(\sigma_2)\mu_{\Delta'} \cup \mu_{\Delta''}.
$$
\end{definition}
\begin{definition}\label{def:coeff_skew}
For $[\Delta] \in \Graphs_n,[\Delta_1], [\Delta_2] \in \Graphs^\sgn_n$ and $(\Delta', \Delta'') \in S(\Delta, \Delta_1, \Delta_2)$ we define the coefficient $\sgn^s{(\Delta,\Delta', \Delta'')} \in \{+1, -1\}$ according to the following formula:
$$
\mu_\Delta = \sgn^s{(\Delta,\Delta', \Delta'')} \sgn(\sigma_1) \sgn_{\Delta_1}(\sigma_1)  \sgn(\sigma_2)\sgn_{\Delta_2}(\sigma_2)\mu_{\Delta'} \cup^s \mu_{\Delta''}.
$$
\end{definition}
\begin{definition}\label{def:coeff_mixed}
	For $[\Delta_1] \in \Graphs_n$, $[\Delta_2], [\Delta] \in \Graphs^\sgn_n$ and $(\Delta', \Delta'') \in S(\Delta, \Delta_1, \Delta_2)$ we define the coefficient $\sgn^t{(\Delta,\Delta', \Delta'')} \in \{+1, -1\}$ according to the following formula:
	$$
	\mu_\Delta = \sgn^t{(\Delta,\Delta', \Delta'')} \sgn_{\Delta_1}(\sigma_1)  \sgn(\sigma_2)\sgn_{\Delta_2}(\sigma_2)\mu_{\Delta'} \cup^t\mu_{\Delta''}.
	$$
\end{definition}

\begin{lemma}\label{lem:sigma_action1}
	The following relations hold:
\begin{align}
&\sgn{(\sigma(\Delta),\Delta',\Delta'')} = \sgn_\Delta(\sigma) \sgn{(\Delta,\Delta', \Delta'')} \label{eq:action_a}\\
&\sgn^s{(\sigma(\Delta),\Delta',\Delta'')} = \sgn_\Delta(\sigma) \sgn^s{(\Delta,\Delta', \Delta'')}\label{eq:action_b}\\
&\sgn^t{(\sigma(\Delta),\Delta',\Delta'')} = \sgn(\sigma)\sgn_\Delta(\sigma) \sgn^t{(\Delta,\Delta', \Delta'')}\label{eq:action_c}
\end{align}
\end{lemma}
\begin{proof}
We notice that in $H^{\bullet}(\Z \Ar_n; \Q)$ we have the equality
\begin{equation}\label{eq:mu_untwisted}
	\mu_\Delta=\sgn{(\Delta,\Delta', \Delta'')} \sigma_1(\mu_{\Delta_1}) \cup \sigma_2(\mu_{\Delta_2}) =  \sgn{(\Delta,\Delta', \Delta'')} \sgn_{\Delta_1}(\sigma_1) \sgn_{\Delta_2}(\sigma_2)\mu_{\Delta'}\cup\mu_{\Delta''}
\end{equation}
and the product
$\sgn{(\Delta,\Delta', \Delta'')} \sgn_{\Delta_1}(\sigma_1)\sgn_{\Delta_2}(\sigma_2)$ is the sign of the shuffle permutation that takes the ordered list of edges of $\Delta'$ followed by the ordered list of edges of $\Delta''$ and gives 
the ordered list of edges of $\Delta$.

Moreover we have the following similar equalities:
\begin{equation}\label{eq:mu_skew}
	\mu_\Delta= \sgn^s{(\Delta,\Delta', \Delta'')} \sigma_1(\mu_{\Delta_1}) \cup^s \sigma_2(\mu_{\Delta_2}) =  \sgn{(\Delta,\Delta', \Delta'')} \sgn(\sigma_1) \sgn_{\Delta_1}(\sigma_1)  \sgn(\sigma_2)\sgn_{\Delta_2}(\sigma_2)\mu_{\Delta'}\cup^s\mu_{\Delta''}.
\end{equation}
and
\begin{equation}\label{eq:mu_twisted}
	\mu_\Delta= \sgn^t{(\Delta,\Delta', \Delta'')} \sigma_1(\mu_{\Delta_1}) \cup^t \sigma_2(\mu_{\Delta_2}) =  \sgn{(\Delta,\Delta', \Delta'')} \sgn_{\Delta_1}(\sigma_1)  \sgn(\sigma_2)\sgn_{\Delta_2}(\sigma_2)\mu_{\Delta'}\cup^t\mu_{\Delta''}.
\end{equation}

Equation \eqref{eq:action_a} follows since $\mu_\Delta = \sgn{(\Delta,\Delta', \Delta'')} \sigma_1(\mu_{\Delta_1}) \cup \sigma_2(\mu_{\Delta_2}) \in H^{\bullet}(\Z \Ar_n; \Q)$ and hence, applying $\sigma$ we get
$$
\mu_{\sigma(\Delta)} = \sgn_\Delta(\sigma) \sigma(\mu_\Delta) =  \sgn_\Delta(\sigma) \sgn{(\Delta,\Delta', \Delta'')}
(\sigma \circ \sigma_1)(\mu_{\Delta_1}) \cup (\sigma \circ \sigma_2)(\mu_{\Delta_2}).
$$

For Equation \eqref{eq:action_b} we can apply $\sigma$ to $\mu_\Delta = \sgn^s{(\Delta,\Delta', \Delta'')} \sigma_1(\mu_{\Delta_1}) \cup^s \sigma_2(\mu_{\Delta_2}) \in H^{\bullet}(\Z \Ar_n; \Q)$ and the equality follows for the same argument, since we have
$$
\mu_{\sigma(\Delta)} = \sgn_\Delta(\sigma) \sigma(\mu_\Delta) =  \sgn_\Delta(\sigma) \sgn^s{(\Delta,\Delta', \Delta'')}
(\sigma \circ \sigma_1)(\mu_{\Delta_1}) \cup^s (\sigma \circ \sigma_2)(\mu_{\Delta_2}).
$$

For Equation \eqref{eq:action_c} we apply $\sigma$ to $\mu_\Delta = \sgn^t{(\Delta,\Delta', \Delta'')} \sgn_{\Delta_1}(\sigma_1)  \sgn(\sigma_2)\sgn_{\Delta_2}(\sigma_2)\mu_{\Delta'} \cup^t \mu_{\Delta''},$
hence we get
\[
\mu_{\sigma(\Delta)} = \sgn(\sigma)\sgn_\Delta(\sigma) \sigma(\mu_\Delta) = \sgn(\sigma) \sgn_\Delta(\sigma) \sgn^t{(\Delta,\Delta', \Delta'')}
(\sigma \circ \sigma_1)(\mu_{\Delta_1}) \cup^t (\sigma \circ \sigma_2)(\mu_{\Delta_2}). 
\qedhere
\]
\end{proof}
%

\begin{definition}
We define the coefficients $a{(\Delta,\Delta_1, \Delta_2)}$, $a^s{(\Delta,\Delta_1, \Delta_2)}$ and $a^t{(\Delta,\Delta_1, \Delta_2)}$ as follows:
\begin{align}
&a{(\Delta,\Delta_1, \Delta_2)} := \sum_{(\Delta', \Delta'') \in  S(\Delta, \Delta_1, \Delta_2)}  \sgn{(\Delta,\Delta', \Delta'')};\\
&a^s{(\Delta,\Delta_1, \Delta_2)} := \sum_{(\Delta', \Delta'') \in  S(\Delta, \Delta_1, \Delta_2)}  \sgn^s{(\Delta,\Delta', \Delta'')};\\
&a^t{(\Delta,\Delta_1, \Delta_2)} := \sum_{(\Delta', \Delta'') \in  S(\Delta, \Delta_1, \Delta_2)}  \sgn^t{(\Delta,\Delta', \Delta'')};
\end{align}

\end{definition}

\begin{lemma}\label{lem:sigma_action2}
\begin{align}
&	a{(\sigma(\Delta),\Delta_1, \Delta_2)}  = \sgn_\Delta(\sigma) a{(\Delta,\Delta_1, \Delta_2)};\\
&	a^s{(\sigma(\Delta),\Delta_1, \Delta_2)}  = \sgn(\sigma) \sgn_\Delta(\sigma) a^s{(\Delta,\Delta_1, \Delta_2)};\\
&	a^t{(\sigma(\Delta),\Delta_1, \Delta_2)}  =\sgn(\sigma) \sgn_\Delta(\sigma) a^t{(\Delta,\Delta_1, \Delta_2)}.
\end{align}

\end{lemma}
\begin{proof}
	This follows immediately from Lemma \ref{lem:sigma_action1}.
\end{proof}

\begin{definition}
We define the direct sum $\R := \Q \oplus \qsgn$ of the trivial $\W_n$-representation and the sign representation. We endow the $\W_n$-module $\R$ with the product structure given by $$(a,a') \cdot (b,b') \mapsto (ab+a'b',ab'+a'b).$$ 
\end{definition}
The $\W_n$-module $\R$ is a $\Z_2$-graded ring with a $\W_n$-covariant product.
We can consider the $\Z_2$-graded $\W_n$-representation $H^\bullet(\Z \Ar_n; \R) = H^\bullet(\Z \Ar_n; \Q) \oplus ( H^\bullet(\Z \Ar_n; \qsgn ))$. This is also naturally an algebra, with an $\W_n$-covariant product.


Taking the restriction to the $\W_n$-invariant submodule 
we get a product that restricts as follows:
\begin{align}
(H^i(\Z \Ar_n; \Q))^{\W_n} \otimes  (H^j(\Z \Ar_n; \Q))^{\W_n} &\mapsto (H^{i+j}(\Z \Ar_n; \Q))^{\W_n};\\
(H^i(\Z \Ar_n; \qsgn))^{\W_n} \otimes(H^j(\Z \Ar_n; \qsgn))^{\W_n}  &\mapsto ( H^{i+j}(\Z \Ar_n; \Q)^{\W_n};\\
(H^i(\Z \Ar_n; \qsgn))^{\W_n} \otimes(  H^j(\Z \Ar_n; \Q))^{\W_n}  &\mapsto ( H^{i+j}(\Z \Ar_n; \qsgn)^{\W_n};\\
( H^i(\Z \Ar_n; \Q))^{\W_n} \otimes(  H^j(\Z \Ar_n; \qsgn))^{\W_n}  &\mapsto ( H^{i+j}(\Z \Ar_n; \qsgn)^{\W_n}.
\end{align}
The cohomology groups described previously fit in the uniform setting of a cohomology ring:
$$H^\bullet(\Gamma_n;\R) = H^\bullet(\Gamma_n;\Q)  \oplus H^\bullet(\Gamma_n;\qsgn).$$

\begin{definition}
We consider the projection 
$\pi_{\W_n}: H^{\bullet}(\Z \Ar_n; \Q) \to H^{\bullet}(\Z \Ar_n; \Q)^{\W_n}$
defined by $$\pi_{\W_n}(\mu_\Delta) := \dfrac{1}{|\W_n|}\sum_{\sigma \in \W_n} \sigma(\mu_\Delta) = 
\dfrac{1}{|\W_n|}\sum_{\sigma \in \W_n} \sgn_\Delta(\sigma)\mu_{\sigma(\Delta)} =
\dfrac{|\StabWn(\Delta)|}{n!}  \alpha_{\Delta}$$ and the projection
$\pi_{\W_n}:  H^{\bullet}(\Z \Ar_n; \qsgn) \to ( H^{\bullet}(\Z \Ar_n; \qsgn))^{\W_n}$
defined by $$\pi_{\W_n}(\mu_\Delta) := \dfrac{1}{|\W_n|}\sum_{\sigma \in \W_n} \sigma(\mu_\Delta) =
\dfrac{1}{|\W_n|}\sum_{\sigma \in \W_n} \sgn(\sigma)\sgn_\Delta(\sigma)\mu_{\sigma(\Delta)} =
\dfrac{|\StabWn(\Delta)|}{n!}  \alpha_{\Delta}^\sgn.$$ 
\end{definition}
\begin{remark} For $\mu_\Delta \in H^{\bullet}(\Z \Ar_n; \Q)$ the projection 
	$\pi_{\W_n}(\mu_\Delta)$
	 is zero if and only if $\Delta \notin \Graphs_n$ (see the proof of Theorem \ref{thm:generators}), while for $\mu_\Delta \in H^{\bullet}(\Z \Ar_n; \qsgn)$ the projection
	 $\pi_{\W_n}(\mu_\Delta)$ is zero if and only if $\StabWn(\Delta)$ contains a permutation $\sigma$ such that $\sgn(\sigma)\sgn_\Delta(\sigma)=-1$, that is equivalent to say that $\Delta \notin \Graphs_n^\sgn$.
\end{remark}


\begin{remark}
	Since the multiplicative structure of $H^{\bullet}(\Z \Ar_n; \Q)$ is $\W_n$-covariant,  given two classes $\omega, \lambda \in H^{\bullet}(\Z \Ar_n; \Q)$ we have that $\pi_{\W_n}(\omega \cup \pi_{\W_n}(\lambda)) = \pi_{\W_n}(\omega)\cup  \pi_{\W_n}(\lambda)$. In particular the product $\pi_{\W_n}(\omega)\cup  \pi_{\W_n}(\lambda)$ is $\W_n$-invariant.
\end{remark}

We can rewrite the definition of $\alpha_\Delta$ given in Theorem \ref{thm:generators}, Equation \eqref{eq:alpha_graph} as follows:
\begin{equation}\label{eq:alpha_delta}
\alpha_\Delta= \sum_{[\sigma] \in \W_n/\StabWn(\Delta)} \sigma(\mu_\Delta) = \frac{n!}{|\StabWn(\Delta)|} \pi_{\W_n}(\mu_\Delta)
\end{equation}
where for any coset $[\sigma] \in \W_n/\StabWn(\Delta)$ the permutation $\sigma$ is a representative of $[\sigma]$ and since $\Delta \in \Graphs_n$ the element $\sigma(\mu_\Delta) \in H^{\bullet}(\Z \Ar_n; \Q)$ does not depend on the choice of the representative.

In a similar way we can rewrite the definition of $\alpha_\Delta^\sgn$ given in Theorem \ref{thm:skewgenerators}, Equation \eqref{eq:alpha_sgn}:
\begin{equation}\label{eq:alpha_delta_sgn}
\alpha_\Delta^\sgn= \sum_{[\sigma] \in \W_n/\StabWn(\Delta)} \sigma(\mu_\Delta) = \frac{n!}{|\StabWn(\Delta)|} \pi_{\W_n}(\mu_\Delta)
\end{equation}
where we consider $\mu_\Delta \in H^{\bullet}(\Z \Ar_n; \qsgn)$. 


\begin{theorem} \label{theo:product}
The multiplicative structure of $H^\bullet(\Z \Ar_n; \R)$ can be computed as follows.
\begin{align}
\intertext{\begin{enumerate}\item[a)]
For $\graph_1, \graph_2 \in \Graphs_n$, the usual cup product in $H^\bullet(\Gamma_n; \Q)$ is given by the formula:\end{enumerate}}
\alpha_{\graph_1} \cup \alpha_{\graph_2} 
&=  \sum_{\graph \in \Graphs_n} a({\Delta_\graph}, \Delta_{\graph_1}, \Delta_{\graph_2}) \alpha_{\graph}
\in H^\bullet(\Gamma_n; \Q);&&
\label{eq:triv_triv}
\\
\intertext{\begin{enumerate}\item[b)] for $\graph_1 \in \Graphs_n, \graph_2 \in \Graphs_n^{\sgn}$, 
the natural bilinear form
$H^{\bullet}(\Gamma_n;\Q) \otimes H^\bullet(\Gamma_n;\qsgn) \to H^{\bullet}(\Gamma_n;\qsgn)$
is given by the formula:\end{enumerate}}
\alpha_{\graph_1} \cup^t \alpha_{\graph_2}^{\sgn} 
& =  \sum_{\graph \in \widetilde{\Graphs}_n} a^t(\Delta_\graph, \Delta_{\graph_1}, \Delta_{\graph_2}) 
\alpha_{\graph}^\sgn
	\in H^\bullet(\Gamma_n; \qsgn);&&
\label{eq:triv_mixodd}\\
\intertext{\begin{enumerate}\item[c)] for $\graph_1, \graph_2 \in \Graphs_n^{\sgn}$, the natural bilinear form $H^\bullet(\Gamma_n;\qsgn) \otimes H^\bullet(\Gamma_n;\qsgn) \to H^\bullet(\Gamma_n;\Q)$ is given by the formula:\end{enumerate}}
\alpha_{\graph_1}^{\sgn} \cup^s \alpha_{\graph_2}^{\sgn} 
&=  \sum_{\graph \in \Graphs_n^\sgn} a^s(\Delta_\graph, \Delta_{\graph_1}, \Delta_{\graph_2}) \alpha_{\graph}
\in H^\bullet(\Gamma_n; \Q). 
\label{eq:twist_twist}
\end{align}
\end{theorem}
%
\begin{proof} The proof is analogous for the three cases. For conciseness we will write $\Delta_1$ for $\Delta_{\graph_1}$ and $\Delta_2$ for $\Delta_{\graph_2}$.
\begin{enumerate} 
\item[a)] For $\graph_1, \graph_2 \in \Graphs_n$ we have:
\begin{align*}
\alpha_{\graph_1} \cup \alpha_{\graph_2}   = & \left( 
\sum_{[\sigma'] \in \W_n/\StabWn(\Delta_1)} \sigma'(\mu_{\Delta_1}) \right) \cup \left( 
\sum_{[\sigma''] \in \W_n/\StabWn(\Delta_2)} \sigma''(\mu_{\Delta_2}) \right) =\\
 = & \left( 
\sum_{[\sigma'] \in \W_n/\StabWn(\Delta_1)} \!\!\! \sgn_{\Delta_1}(\sigma') \mu_{\sigma'(\Delta_1)} \right) \cup \left( 
\sum_{[\sigma''] \in \W_n/\StabWn(\Delta_2)}\!\!\! \sgn_{\Delta_2}(\sigma'') \mu_{\sigma''(\Delta_2)} \right) =\\
= & \sum_{\graph \in \mathcal{G}_n} \sum_{[\sigma] \in \W_n/\StabWn(\Delta_{\graph})} \!\!\! \sum_{\substack{
[\sigma'] \in \W_n/\StabWn(\Delta_1)\\
[\sigma''] \in \W_n/\StabWn(\Delta_2)\\
E\sigma'(\Delta_1) \sqcup E\sigma''(\Delta_2)=E\sigma(\Delta_{\graph})
}}
\!\!\!\!\!\!\sgn_{\Delta_1}(\sigma') \mu_{\sigma'(\Delta_1)} 
\sgn_{\Delta_2}(\sigma'') \mu_{\sigma''(\Delta_2)} =\\
= & \sum_{\graph \in \mathcal{G}_n} \sum_{[\sigma] \in \W_n/\StabWn(\Delta_{\graph})} \!\!\! \sum_{\substack{
		[\sigma'] \in \W_n/\StabWn(\Delta_1)\\
		[\sigma''] \in \W_n/\StabWn(\Delta_2)\\
		E\sigma'(\Delta_1) \sqcup E\sigma''(\Delta_2)=E\sigma(\Delta_{\graph})
}}
\!\!\!\!\!\! \sgn{(\sigma(\Delta_\graph),\sigma'(\Delta_1),\sigma''(\Delta_2))}\mu_{\sigma(\Delta_\graph)}
=\\
= & \sum_{\graph \in \mathcal{G}_n} \sum_{[\sigma] \in \W_n/\StabWn(\Delta_{\graph})} \!\!\! 
a(\sigma(\Delta_\graph),\Delta_1,\Delta_2) \mu_{\sigma(\Delta_\graph)}
 =\\
 = & \sum_{\graph \in \mathcal{G}_n} \sum_{[\sigma] \in \W_n/\StabWn(\Delta_{\graph})} \!\!\! 
 a(\Delta_\graph,\Delta_1,\Delta_2) \sgn_{\Delta_\graph}(\sigma) \mu_{\sigma(\Delta_\graph)}
 =\\
 = & \sum_{\graph \in \Graphs_n} \sum_{[\sigma] \in \W_n/\StabWn(\Delta_{\graph})} \!\!\! 
 a(\Delta_\graph,\Delta_1,\Delta_2) \sigma(\mu_{\Delta_\graph})
 =\\
 = & \sum_{\graph \in \Graphs_n}  a(\Delta_\graph,\Delta_1,\Delta_2) \alpha_{\graph}.
\end{align*}
\item[b)]
 For $\graph_1 \in \Graphs_n, \graph_2 \in \Graphs_n^{\sgn}$
 we have:
\begin{align*}
\alpha_{\graph_1} \cup^t \alpha^\sgn_{\graph_2}   = & \left( 
\sum_{[\sigma'] \in \W_n/\StabWn(\Delta_1)} \sigma'(\mu_{\Delta_1}) \right) \cup^t \left( 
\sum_{[\sigma''] \in \W_n/\StabWn(\Delta_2)} \sigma''(\mu_{\Delta_2}) \right) =\\
= & \left( 
\sum_{[\sigma'] \in \W_n/\StabWn(\Delta_1)} \!\!\!\! \sgn_{\Delta_1}(\sigma') \mu_{\sigma'(\Delta_1)} \right) \cup^t \left( 
\sum_{[\sigma''] \in \W_n/\StabWn(\Delta_2)}\!\!\!\!\! \sgn(\sigma'')\sgn_{\Delta_2}(\sigma'') \mu_{\sigma''(\Delta_2)} \right) =\\
= & \sum_{\graph \in \mathcal{G}_n} \sum_{[\sigma] \in \W_n/\StabWn(\Delta_{\graph})} \!\!\! \sum_{\substack{
		[\sigma'] \in \W_n/\StabWn(\Delta_1)\\
		[\sigma''] \in \W_n/\StabWn(\Delta_2)\\
		E\sigma'(\Delta_1) \sqcup E\sigma''(\Delta_2)=E\sigma(\Delta_{\graph})
}}
\!\!\!\!\!\!\!\!\!\!\sgn_{\Delta_1}(\sigma') \sgn(\sigma'')\sgn_{\Delta_2}(\sigma'')\mu_{\sigma'(\Delta_1)}\!\! \cup^t\!\! \mu_{\sigma''(\Delta_2)} =\\
= & \sum_{\graph \in \mathcal{G}_n} \sum_{[\sigma] \in \W_n/\StabWn(\Delta_{\graph})} \!\!\! \sum_{\substack{
		[\sigma'] \in \W_n/\StabWn(\Delta_1)\\
		[\sigma''] \in \W_n/\StabWn(\Delta_2)\\
		E\sigma'(\Delta_1) \sqcup E\sigma''(\Delta_2)=E\sigma(\Delta_{\graph})
}}
\!\!\!\!\!\!\!\!\! \sgn^t{(\sigma(\Delta_\graph),\sigma'(\Delta_1),\sigma''(\Delta_2))}\mu_{\sigma(\Delta_\graph)}
=\\
= & \sum_{\graph \in \mathcal{G}_n} \sum_{[\sigma] \in \W_n/\StabWn(\Delta_{\graph})} \!\!\! 
a^t(\sigma(\Delta_\graph),\Delta_1,\Delta_2) \mu_{\sigma(\Delta_\graph)}
=\\
= & \sum_{\graph \in \mathcal{G}_n} \sum_{[\sigma] \in \W_n/\StabWn(\Delta_{\graph})} \!\!\! 
a^t(\Delta_\graph,\Delta_1,\Delta_2) \sgn(\sigma) \sgn_{\Delta_\graph}(\sigma) \mu_{\sigma(\Delta_\graph)}
=\\
= & \sum_{\graph \in \widetilde{\Graphs}_n} \sum_{[\sigma] \in \W_n/\StabWn(\Delta_{\graph})} \!\!\! 
a^t(\Delta_\graph,\Delta_1,\Delta_2) \sigma(\mu_{\Delta_\graph})
=\\
= & \sum_{\graph \in \widetilde{\Graphs}_n}  a^t(\Delta_\graph,\Delta_1,\Delta_2) \alpha_{\graph}^\sgn.
\end{align*}
\item[c)] For $\graph_1, \graph_2 \in \Graphs_n^{\sgn}$ we have:
\begin{align*}
\alpha^\sgn_{\graph_1} \cup^s \alpha^\sgn_{\graph_2}   = & \left( 
\sum_{[\sigma'] \in \W_n/\StabWn(\Delta_1)} \sigma'(\mu_{\Delta_1}) \right) \cup^s \left( 
\sum_{[\sigma''] \in \W_n/\StabWn(\Delta_2)} \sigma''(\mu_{\Delta_2}) \right) =\\
= & \left( 
\sum_{[\sigma'] \in \W_n/\StabWn(\Delta_1)} \!\!\!\!\!\!\!\! \sgn(\sigma') \sgn_{\Delta_1}\!(\sigma') \mu_{\sigma'(\Delta_1)} \right) 
\! \cup^s \! \left( 
\sum_{[\sigma''] \in \W_n/\StabWn(\Delta_2)}\!\!\!\!\!\!\!\!\! \sgn(\sigma'')\sgn_{\Delta_2}\!(\sigma'') \mu_{\sigma''(\Delta_2)} \right) =\\
= & \!\!\!\sum_{\substack{
\graph \in \mathcal{G}_n\\
[\sigma] \in \W_n/\StabWn(\Delta_{\graph})
}}
\!\!\!\!\!\!\!\! \sum_{\substack{
		[\sigma'] \in \W_n/\StabWn(\Delta_1)\\
		[\sigma''] \in \W_n/\StabWn(\Delta_2)\\
		E\sigma'(\Delta_1) \sqcup E\sigma''(\Delta_2)=E\sigma(\Delta_{\graph})
}}
\!\!\!\!\!\!\!\!\!\!\!\!\!\!\!\!  \sgn(\sigma') \sgn_{\Delta_1}\!(\sigma')  
 \sgn(\sigma'')\sgn_{\Delta_2}\!(\sigma'') \mu_{\sigma'(\Delta_1)} \cup^s\mu_{\sigma''(\Delta_2)} =\\
= & \sum_{\graph \in \mathcal{G}_n} \sum_{[\sigma] \in \W_n/\StabWn(\Delta_{\graph})} \!\!\! \sum_{\substack{
		[\sigma'] \in \W_n/\StabWn(\Delta_1)\\
		[\sigma''] \in \W_n/\StabWn(\Delta_2)\\
		E\sigma'(\Delta_1) \sqcup E\sigma''(\Delta_2)=E\sigma(\Delta_{\graph})
}}
\!\!\!\!\!\!\!\!\! \sgn^s{(\sigma(\Delta_\graph),\sigma'(\Delta_1),\sigma''(\Delta_2))}\mu_{\sigma(\Delta_\graph)}
=\\
= & \sum_{\graph \in \mathcal{G}_n} \sum_{[\sigma] \in \W_n/\StabWn(\Delta_{\graph})} \!\!\! 
a^s(\sigma(\Delta_\graph),\Delta_1,\Delta_2) \mu_{\sigma(\Delta_\graph)}
=\\
= & \sum_{\graph \in \mathcal{G}_n} \sum_{[\sigma] \in \W_n/\StabWn(\Delta_{\graph})} \!\!\! 
a^s(\Delta_\graph,\Delta_1,\Delta_2) \sgn(\sigma) \sgn_{\Delta_\graph}(\sigma) \mu_{\sigma(\Delta_\graph)}
=\\
= & \sum_{\graph \in \widetilde{\Graphs}_n} \sum_{[\sigma] \in \W_n/\StabWn(\Delta_{\graph})} \!\!\! 
a^s(\Delta_\graph,\Delta_1,\Delta_2) \sigma(\mu_{\Delta_\graph})
=\\
= & \sum_{\graph \in \widetilde{\Graphs}_n}  a^s(\Delta_\graph,\Delta_1,\Delta_2) \alpha_{\graph}. \qedhere
\end{align*}
\end{enumerate}
\end{proof}
\subsection{Poincar\'e duality} Recall that we write $\full_n$ for the complete graph on $n$ vertices. It has $n(n-1)/2$ edges. It is easily checked that
any transposition acts with parity $(-1)^{n-2}$ on the set of edges, therefore every permutation acts with the same
parity on the vertices and on the edges when $n$ is odd, and acts evenly on the edges when $n$ is even (in other terms, $\sgn_{\full_n}(\sigma) = \sgn(\sigma)^n$ for all $\sigma \in \W_n$).
In particular we have the following result.
\begin{proposition}	Let $\full_n$ be the full graph on $n$ vertices. 
\begin{enumerate}
	\item $\full_n$ is invariant if and only if $n$ is even; 
	\item $\full_n$ is skew-invariant if and only if $n$ is odd.
\end{enumerate}
\end{proposition} 

Since $\full_n$ is the only graph on $n$ vertices with $n(n-1)/2$ edges, we get the following application.
\begin{corollary} \label{cor:top} Let $N:= \frac{n(n-1)}{2}$.
\begin{enumerate}
\item $H^N(\Gamma_n;\Q) = \Q$ if and only if $n$ is even, and $H^N(\Gamma_n;\Q) = 0$ otherwise;
\item $H^N(\Gamma_n;\qsgn)= \Q$ if and only if $n$ is odd, and $H^N(\Gamma_n;\qsgn)= 0$ otherwise.
\end{enumerate}
\end{corollary}

\begin{lemma}  \label{lem:pairing} Let $G$ be a finite group, $\kk$ a field of characteristic $0$ with trivial $G$-action, $U,V$ two $\kk G$-modules.
	If $\varphi : U \otimes_{\kk} V \to \kk$ is a non-degenerate $G$-equivariant pairing, then
	it induces a non-degenerate pairing $U^G \otimes_{\kk} V^G \to \kk$. 
\end{lemma}
\begin{proof}
	From the inclusion $U^G \otimes V^G \subset U \otimes V$ we deduce from $\varphi$ a linear map $\varphi_G : U^G \otimes V^G \to \kk$.
	Let $x \in U^G$ with $x \neq 0$. By assumption there exists $y\in V$ with $\varphi(x \otimes y) \neq 0$. Then, for
	every $g \in G$ we have $\varphi(x \otimes g.y) = \varphi(g.(x \otimes y)) = g.\varphi(x \otimes y) = \varphi(x \otimes y)$
	hence $\varphi(x \otimes \hat{y}) = \varphi(x\otimes y) \neq 0$ if $\hat{y} = (1/|G|)\sum_{g \in G} g.y$. Since $\hat{y} \in V^G \setminus \{ 0 \}$
	and $\varphi_G(x \otimes \hat{y}) \neq 0$,
	this proves the claim.
\end{proof}

\begin{proposition}
	If $n$ is even there are nondegenerate pairings
	\begin{align}
	H^r(\Gamma_n;\Q) \otimes H^{N-r}(\Gamma_n;\Q) & \to H^N(\Gamma_n;\Q) \simeq \Q;\\
	H^r(\Gamma_n;\qsgn) \otimes H^{N-r}(\Gamma_n;\qsgn) & \to H^N(\Gamma_n;\Q) \simeq \Q;\\
\intertext{while if $n$ is odd there is a nondegenerate pairing}
		H^r(\Gamma_n;\Q) \otimes H^{N-r}(\Gamma_n;\qsgn) & \to H^N(\Gamma_n;\qsgn) \simeq \Q.
	\end{align}
\end{proposition}
\begin{proof}
The result is a direct application of Lemma \ref{lem:pairing} to the case considered in Corollary \ref{cor:top}.
When the groups  $H^N(\Gamma_n;\Q)$ and  $H^N(\Gamma_n;\qsgn)$ 
are non-trivial, the non-degenerate duality pairings
\begin{align}
	\Lambda^r(\Q\Ar_n) &\otimes \Lambda^{N-r}(\Q \Ar_n) && \to \Lambda^N(\Q \Ar_n) \simeq \Q \label{pairing:norm}\\
	(\qsgn\otimes \Lambda^r(\Q\Ar_n)) &\otimes (\qsgn \otimes \Lambda^{N-r}(\Q \Ar_n)) &&\to  \Lambda^N(\Q \Ar_n) \simeq \Q\label{pairing:skew}\\
	\Lambda^r(\Q\Ar_n) &\otimes (\qsgn \otimes \Lambda^{N-r}(\Q \Ar_n)) &&\to \qsgn \otimes \Lambda^N(\Q \Ar_n) \simeq \Q \label{pairing:twist}
\end{align}
induces non-degenerate Poincar\'e duality pairings on the cohomology of $\Gamma_n$.
\end{proof}

For $\Delta$ a graph on $n$ vertices, we denote by $\Delta^c$ its complement inside $\full_n$. This provides a combinatorial
description of the corresponding Poincar\'e duality pairings, which can be checked directly.

\begin{theorem} \label{thm:poincare} Let $N = n(n-1)/2$. Then the map $\Delta \mapsto \Delta^c$
induces the following isomorphisms:
\begin{enumerate}
\item  \label{itm:even_untwisted} If $n$ is even, $H^r(\Gamma_n;\Q) \to H^{N-r}(\Gamma_n;\Q)$;
 \item If $n$ is even  $H^r(\Gamma_n;\qsgn) \to H^{N-r}(\Gamma_n;\qsgn)$;
\item \label{itm:odd}  If $n$ is odd $H^r(\Gamma_n;\qsgn) \to H^{N-r}(\Gamma_n;\Q)$.
\end{enumerate}

\end{theorem}
\begin{proof} The results follows applying Lemma \ref{lem:pairing} to the pairings \eqref{pairing:norm}, \eqref{pairing:skew}, \eqref{pairing:twist} in the cases where Corollary \ref{cor:top} gives non-trivial groups. The pairing turns out to be a special case of Theorem \ref{theo:product}. In particular we have the following cases:
\begin{enumerate}
\item 	Let $\Delta$ be a subgraph of $\full_n$ with $n$ vertices, and $\sigma$ an automorphism of $\Delta$.
Then $\sigma$ is also an automorphism of $\Delta^c$, and we have $\sgn_\Delta(\sigma)\sgn_{\Delta^c}(\sigma)
= \sgn_{\full_n}(\sigma) = \sgn(\sigma)^{n-2}$.
If $n$ is even, we get $\sgn_\Delta(\sigma) = \sgn_{\Delta^c}(\sigma)$ and this yields that $\Delta \mapsto \Delta^c$ maps invariant graphs to invariant graphs, whence (1).
\item 
If $n$ is even, we get $\sgn_\Delta(\sigma) \sgn(\sigma)  \sgn_{\Delta^c}(\sigma)\sgn(\sigma)= \sgn_{\full_n}(\sigma) \sgn(\sigma)^2=1$. This implies $\sgn_\Delta(\sigma) \sgn(\sigma)= \sgn_{\Delta^c}(\sigma)\sgn(\sigma)$ and
hence $\Delta \mapsto \Delta^c$ maps skew-invariant graphs to skew-invariant graphs, whence (2).
\item 
If $n$ is odd, $\sgn(\sigma)\sgn_{\Delta}(\sigma) \sgn_{\Delta^c}(\sigma) = \sgn(\sigma)\sgn_{\full_n}(\sigma) = \sgn(\sigma)^{1+n-2}=1$. This imples that $\sgn(\sigma)\sgn_{\Delta}(\sigma) = \sgn_{\Delta^c} $ 
and hence $\Delta \mapsto \Delta^c$ maps skew-invariant graphs to invariant graphs, whence (3). \qedhere
\end{enumerate}	
\end{proof}

\begin{remark}
We notice that 
for a skew-invariant graph $\Delta \subset \full_n$ with $n$ odd and 
we can consider the graph $\Sigma\Delta$ given by embedding $\Delta$ as a subgraph of $\full_{n+2}$ and then adding all possible edges $(i,j)$ with $i \in \{1, \ldots, n\}$ and $j \in \{n+1, n+2\}$. The graph $\Sigma\Delta$ is skew-invariant.
\end{remark}

\subsection{Representation-theoretic description}

\begin{proposition}\label{prop:dim_not_skew}
	Let us assume $n \geq 4$ and let $b_{n,r} = \dim H^{r}(\Z \Ar_n; \Q)^{\W_n}$.   We have:
\begin{equation}\label{eq:dim_not_skew}
	b_{n,r}=
	\left(\sum_{p+q = r} \left\langle  [n-p,p] , \bigwedge^q [n-2,2] \right\rangle \right)
	\oplus \left(
	\sum_{p+q = r-1}\left\langle  [n-p,p], \bigwedge^q [n-2,2] \right\rangle \right)
\end{equation}
where $\langle \lambda, \mu \rangle$ denotes the scalar product of characters for $\W_n$.
\end{proposition}

\begin{proof}
As a representation of $\W_n$, $\Q \Ar_n$ can be decomposed into irreducibles as
$[n-2,2] + [n-1,1]
+[n]
$ (see for example \cite{church_farb}). We have
$$
H^{r}(\Z \Ar_n; \Q) \simeq \bigwedge^r \Q \Ar_n \simeq \bigwedge^r\left(
[n-2,2] + [n-1,1]+[n] \right) \simeq \bigoplus_{a+b+c = r} \bigwedge^a [n-2,2] \otimes  \bigwedge^b [n-1,1] \otimes
\bigwedge^c [n]
$$
Since $\bigwedge^{\bullet} [n] = \bigwedge^{\bullet} \Q
= \bigwedge^{0} \Q + \bigwedge^1 \Q$
and $\bigwedge^p [n-1,1] = [n-p,1^p]$
this yields
$$
H^{r}(\Z \Ar_n; \Q) \simeq 
\left(\bigoplus_{p+q = r} [n-p,1^p] \otimes \bigwedge^q [n-2,2] 
\right) \oplus \left(
\bigoplus_{p+q = r-1}[n-p,1^p] \otimes \bigwedge^q [n-2,2] \right)
$$
with the convention $[n-p,1^p] = 0$ if $p \geq n$ or $p<0$. The result follows from the last equation.
\end{proof}

The Proposition above provides a reasonably fast algorithm to compute the first Betti numbers (see Table \ref{table:firstbetti}),
and the first terms of the Poincar\'e series for $\Gamma_{\infty}$. Computer computations of the formula in Equation \eqref{eq:dim_not_skew} give the first terms of the stable Poincar\'e series:
\begin{theorem} The Poincar\'e series for $\Gamma_{\infty}$ starts with
\begin{equation}
\begin{split}
1+ t+ t^4 +7t^5 +17t^6 +30t^7+ 88t^8+ 335t^9+ 1143t^{10}+ 3866t^{11} +  
 14289t^{12} + 56557t^{13}+\\+ 231012t^{14}+ 971537t^{15}+ 4238570t^{16}+ \ldots 
\end{split}	
\end{equation}
\end{theorem}
One also observes the existence of a factor $(1+t)^m$ of the Poincar\'e polynomial
of $H^\bullet(\Gamma_{n};\Q)$ for small $n$, in Table \ref{table:gamfactpol}. For $n = 9,10$, the value of $m$ is $4,5$. 

\afterpage{%
	\begin{landscape}
		\centering 
$$
\begin{array}{|c||c|c|c|c|c|c|c|c|c|c|c|c|c|c|c|c|c|c|c|c|c|}
\hline
H^r(\Gamma_n;\Q) & 0 & 1 & 2 & 3 & 4 & 5 & 6 & 7 & 8 & 9 & 10 & 11 & 12 & 13 & 14 & 15 & 16 & 17 & 18 & 19 & 20 \\
\hline
\hline
\Gamma_4 & 1 & 1 & 0 & 0 & 0 & 1 & 1 & & & & & & & & & &  & & & & \\ 
\hline
\Gamma_5 & 1 & 1 & 0 & 0 & 1 & 4 & 4 & 1 & 0 & 0 & 0 &  & & & & & & & & &  \\ 
\hline
\Gamma_6 & 1 & 1 & 0 & 0 & 1 & 6 & 10 & 9 & 9 & 10 & 6 & 1 & 0 & 0 & 1 & 1 & & & & &  \\ 
\hline
\Gamma_7 & 1 & 1 & 0 & 0 & 1 & 7 & 15 & 20 & 37 & 72 & 88 & 71 & 48 & 35 & 32 & 22 & 6 & 0 & 0 &0 & \ldots  \\ 
\hline
\Gamma_8 & 1 & 1 & 0 & 0 & 1 & 7 & 17 & 28 & 67 & 182 & 364 & 566 & 767 & 936 & 1006 & 936 & 767 & 566 & 364 & 182 & \dots \\
\hline
 \Gamma_9\footnotemark[1] & 1 &  1 & 0& 0 & 1 & 7 & 17 & 30&  84&  278&  738 &1673& 3499& 6619& 10855& 15464& 19862& 23572& 25458 & 24285 & \dots \\ 
 \hline
\end{array}
$$
\captionof{table}{First Betti numbers for $\Gamma_n$}
\label{table:firstbetti}

	$$
	\begin{array}{|c|l|}
		\hline
		\Gamma_3 & t+1  \\ 
		\hline
		\Gamma_4 & (t+1)^2(t^4 - t^3 + t^2 - t + 1) \\ 
		\hline
		\Gamma_5 & (t+1)^2( t^5 + 2t^4 - t^3 + t^2 - t + 1) \\
		\hline
		\Gamma_6 &  (t+1)^3(t^{12} - 2t^{11} + 3t^{10} - 4t^9 + 6t^8 - 3t^7 + 5t^6 - 
		3t^5 + 6t^4 - 4t^3 + 3t^2 - 2t + 1)\\
		\hline
		\Gamma_7 &  (t+1)^3(6t^{13} + 4t^{12} + 2t^{11} + 11t^{10} + 5t^9 + 21t^8 - t^7 + 7t^6 - 2t^5 + 
		6t^4 - 4t^3 + 3t^2 - 2t + 1)\\
		\hline
		\Gamma_8 &  (t+1)^4(t^{24} - 3t^{23} + 6t^{22} - 10t^{21} + 16t^{20} - 18t^{19} + 27t^{18} - 26t^
		{17} + 65t^{16} - 12t^{15} + 99t^{14} 
		+ 8t^{13} + 124t^{12} + \dots + 16t^4 - 10t^3 + 
		6t^2 - 3t + 1) \\
		\hline        
	\end{array}
	$$
	\captionof{table}{Factorized Poincar\'e polynomials for $\Gamma_n$}
	\label{table:gamfactpol}

	$$
	\begin{array}{|c||c|c|c|c|c|c|c|c|c|c|c|c|c|c|c|c|c|c|c|c|c|}
		\hline
		H^r(\Gamma_n;\qsgn) & 0 & 1 & 2 & 3 & 4 & 5 & 6 & 7 & 8 & 9 & 10 & 11 & 12 & 13 & 14 & 15 & 16 & 17 & 18 & 19 & 20 \\
		\hline
		\Gamma_4 & 0 & 0 & 1 & 2 & 1 & 0 & 0 & & & & & & & & & &   & & &   &\\ 
		\hline 
		\Gamma_5 & 0 & 0 & 0 & 1 & 4 & 4 & 1 & 0 & 0 & 1 & 1 & & & & & & & & &   &\\ 
		\hline 
		\Gamma_6 & 0 & 0 & 0 & 0 & 3 & 9 & 10 & 6 & 6 & 10 & 9 & 3 & 0 & 0 & 0 & 0 & && &   & \\ 
		\hline 
		\Gamma_7 & 0 & 0 & 0 & 0 & 0 & 6 & 22 & 32 & 35 & 48 & 71 & 88 & 72 & 37 & 20 & 15 & 7 & 1 & 0 & 0  & \ldots \\ 
		\hline 
		\Gamma_8 & 0 & 0 & 0 & 0 & 0 & 1 & 16 & 53 & 97 & 160 & 301 & 551 & 815 & 955 & 982 & 955 & 815 &551& 301 & 160  &  \dots  \\
		\hline
		\Gamma_9\footnotemark[2] & 0 & 0& 0& 0& 0& 0& 3& 34& 135& 327& 716& 1637& 3525& 6559& 10549& 15282& 20325& 
		 24285 & 25458 & 23572 & \dots \\ 
		\hline
	\end{array}
	$$
	\captionof{table}{First skew Betti numbers for $\Gamma_n$}
	\label{table:firstskewbetti}
\footnotetext[1]{The full list of Betti numbers is:$ 1 $, $  1 $, $ 0$, $ 0 $, $ 1 $, $ 7 $, $ 17 $, $ 30$, $  84$, $  278$, $  738 $, $1673$, $ 3499$, $ 6619$, $ 10855$, $ 15464$, $ 19862$, $ 23572$, $ 25458 $, $ 24285 $, $  20325$, $ 15282$, $ 10549$, $ 6559$, $ 3525$, $ 1637$, $ 716$, $ 327$, $ 135$, $ 34$, $ 3$, $ 0$, $ 0$, $ 0$, $ 0$, $ 0$, $ 0 $}
\footnotetext[2]{The full list of Betti numbers is:$ 0 $, $ 0$, $ 0$, $ 0$, $ 0$, $ 0$, $ 3$, $ 34$, $ 135$, $ 327$, $ 716$, $ 1637$, $ 3525$, $ 6559$, $ 10549$, $ 15282$, $ 20325$, $  24285 $, $ 25458 $, $ 23572 $, $ 19862$, $ 15464$, $ 10855$, $ 6619$, $ 3499$, $ 1673$, $ 738$, $  278$, $ 84$, $ 30$, $ 17$, $ 7$, $ 1$, $ 0$, $ 0$, $ 1$, $ 0$.}
\end{landscape}
\clearpage
}

From the isomorphism $H^{r}(\Z \Ar_n; \qsgn) \simeq \qsgn \otimes H^{r}(\Z \Ar_n; \Q)$
one gets, using the same argument, the analogous of Proposition \ref{prop:dim_not_skew} for the case of skew coefficients.
\begin{proposition}
Let us assume $n \geq 4$ and let $b_{n,r}^\sgn = \dim H^{r}(\Z \Ar_n; \qsgn)^{\W_n}$.   We have:

\begin{equation}\label{eq:skew_betti}
	b_{n,r}^\sgn =
\left(\sum_{p+q = r} \!\left\langle\!  \qsgn \otimes [n-p,1^p] , \bigwedge^q [n-2,2] \!\right\rangle\! \right)\!
 \oplus \!\left(
\sum_{p+q = r-1}\!\left\langle\! \qsgn \otimes  [n-p,1^p], \bigwedge^q [n-2,2] \!\right\rangle \right).
\end{equation}
\end{proposition}
As before this enables us to compute the first skew Betti numbers (see Table \ref{table:firstskewbetti}).

\begin{remark}
The vanishing of $H^{r}(\Z \Ar_n; \qsgn)^{\W_n}$ for $n$ large can also be deduced from Equation \eqref{eq:skew_betti},
as $[n-2,2]$ is a constituent of $U^{\otimes 2}$ where $U$ is the permutation representation $\W_n < \GL_n(\Q)$, hence any constituent of 
$\bigwedge^q [n-2,2] \subset [n-2,2]^{\otimes q}$ is a constituent of $U^{\otimes 2q} = (\mathrm{Ind}_{S_{n-1}}^{\W_n} \mathbbm{1})^{\otimes 2q}$,
and the Young diagram of such a constituent has, by Young's rule, at most $2q$ rows. Since $\qsgn \otimes  [n-p,1^p]$ has $n-p$ rows, this cohomology group
can be nonzero only if $n \leq p+2q$ whenever $p+q \in \{ r-1,r\}$.
\end{remark}

\subsection{Low-dimensional cohomology}

With trivial coefficients, natural bases (up to a sign for each basis vector) are thus given,
for $H^1$ by the only simple graph with 1 edge, for $H^4$ by the linear graph on 5 vertices
and for $H^5$ by the following 7 graphs 

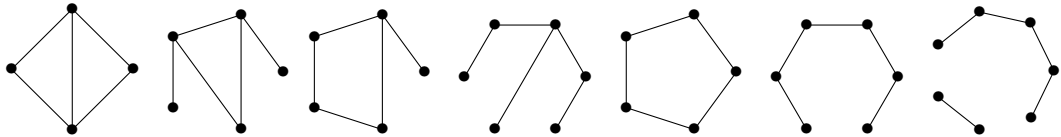
\begin{center}
\begin{tikzpicture}[scale=.4]
\coordinate (A1) at (90:2);
\draw (A1) node {$\bullet$};
\coordinate (A2) at (180:2);
\draw (A2) node {$\bullet$};
\coordinate (A3) at (270:2);
\draw (A3) node {$\bullet$};
\coordinate (A4) at (360:2);
\draw (A4) node {$\bullet$};
\draw (A1) -- (A3);
\draw (A1) -- (A2);
\draw (A1) -- (A4);
\draw (A3) -- (A2);
\draw (A3) -- (A4);
\end{tikzpicture}
\begin{tikzpicture}[scale=.4]
\coordinate (A1) at (72:2);
\draw (A1) node {$\bullet$};
\coordinate (A2) at (144:2);
\draw (A2) node {$\bullet$};
\coordinate (A3) at (216:2);
\draw (A3) node {$\bullet$};
\coordinate (A4) at (288:2);
\draw (A4) node {$\bullet$};
\coordinate (A5) at (360:2);
\draw (A5) node {$\bullet$};
\draw (A1) -- (A2);
\draw (A1) -- (A4);
\draw (A1) -- (A5);
\draw (A2) -- (A4);
\draw (A2) -- (A3);
\end{tikzpicture}\begin{tikzpicture}[scale=.4]
\coordinate (A1) at (72:2);
\draw (A1) node {$\bullet$};
\coordinate (A2) at (144:2);
\draw (A2) node {$\bullet$};
\coordinate (A3) at (216:2);
\draw (A3) node {$\bullet$};
\coordinate (A4) at (288:2);
\draw (A4) node {$\bullet$};
\coordinate (A5) at (360:2);
\draw (A5) node {$\bullet$};
\draw (A1) -- (A2);
\draw (A1) -- (A4);
\draw (A1) -- (A5);
\draw (A2) -- (A3);
\draw (A4) -- (A3);
\end{tikzpicture}
\begin{tikzpicture}[scale=.4]
\coordinate (A1) at (60:2);
\draw (A1) node {$\bullet$};
\coordinate (A2) at (120:2);
\draw (A2) node {$\bullet$};
\coordinate (A3) at (180:2);
\draw (A3) node {$\bullet$};
\coordinate (A4) at (240:2);
\draw (A4) node {$\bullet$};
\coordinate (A5) at (300:2);
\draw (A5) node {$\bullet$};
\coordinate (A6) at (360:2);
\draw (A6) node {$\bullet$};
\draw (A1) -- (A2);
\draw (A1) -- (A6);
\draw (A1) -- (A4);
\draw (A2) -- (A3);
\draw (A6) -- (A5);
\end{tikzpicture}
\begin{tikzpicture}[scale=.4]
\coordinate (A1) at (72:2);
\draw (A1) node {$\bullet$};
\coordinate (A2) at (144:2);
\draw (A2) node {$\bullet$};
\coordinate (A3) at (216:2);
\draw (A3) node {$\bullet$};
\coordinate (A4) at (288:2);
\draw (A4) node {$\bullet$};
\coordinate (A5) at (360:2);
\draw (A5) node {$\bullet$};
\draw (A1) -- (A2);
\draw (A1) -- (A5);
\draw (A2) -- (A3);
\draw (A5) -- (A4);
\draw (A3) -- (A4);
\end{tikzpicture}
\begin{tikzpicture}[scale=.4]
\coordinate (A1) at (60:2);
\draw (A1) node {$\bullet$};
\coordinate (A2) at (120:2);
\draw (A2) node {$\bullet$};
\coordinate (A3) at (180:2);
\draw (A3) node {$\bullet$};
\coordinate (A4) at (240:2);
\draw (A4) node {$\bullet$};
\coordinate (A5) at (300:2);
\draw (A5) node {$\bullet$};
\coordinate (A6) at (360:2);
\draw (A6) node {$\bullet$};
\draw (A1) -- (A2);
\draw (A1) -- (A6);
\draw (A2) -- (A3);
\draw (A6) -- (A5);
\draw (A3) -- (A4);
\end{tikzpicture}
\begin{tikzpicture}[scale=.4]
\coordinate (A1) at (51.999074699:2);
\draw (A1) node {$\bullet$};
\coordinate (A2) at (102.8571428571:2);
\draw (A2) node {$\bullet$};
\coordinate (A3) at (154.9635917385:2);
\draw (A3) node {$\bullet$};
\coordinate (A4) at (205.7142857143:2);
\draw (A4) node {$\bullet$};
\coordinate (A5) at (257.1428571429:2);
\draw (A5) node {$\bullet$};
\coordinate (A6) at (308.5714285714:2);
\draw (A6) node {$\bullet$};
\coordinate (A7) at (360:2);
\draw (A7) node {$\bullet$};
\draw (A1) -- (A2);
\draw (A1) -- (A7);
\draw (A2) -- (A3);
\draw (A7) -- (A6);
\draw (A4) -- (A5);
\end{tikzpicture}
\end{center}

As an example of the cup-product formula, we get the following

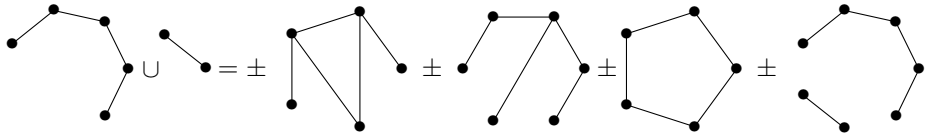
\begin{center}
\begin{tikzpicture}
\begin{scope}[scale=.4]
\coordinate (A1) at (51.999074699:2);
\draw (A1) node {$\bullet$};
\coordinate (A2) at (102.8571428571:2);
\draw (A2) node {$\bullet$};
\coordinate (A3) at (154.9635917385:2);
\draw (A3) node {$\bullet$};
\coordinate (A4) at (205.7142857143:2);
\coordinate (A5) at (257.1428571429:2);
\coordinate (A6) at (308.5714285714:2);
\draw (A6) node {$\bullet$};
\coordinate (A7) at (360:2);
\draw (A7) node {$\bullet$};
\draw (A1) -- (A2);
\draw (A1) -- (A7);
\draw (A2) -- (A3);
\draw (A7) -- (A6);
\end{scope}
\begin{scope}[shift={(1.1,0)}]
\draw (0,0) node {$\cup$};
\end{scope}
\begin{scope}[scale=.4,shift={(5,2)}]
\coordinate (A1) at (51.999074699:2);
\coordinate (A2) at (102.8571428571:2);
\coordinate (A3) at (154.9635917385:2);
\coordinate (A4) at (205.7142857143:2);
\draw (A4) node {$\bullet$};
\coordinate (A5) at (257.1428571429:2);
\draw (A5) node {$\bullet$};
\coordinate (A6) at (308.5714285714:2);
\coordinate (A7) at (360:2);
\draw (A4) -- (A5);
\end{scope}
\begin{scope}[shift={(2.3,0)}]
\draw (0,0) node {$= \pm$};
\end{scope}
\begin{scope}[scale=.4,shift={(9,0)}]
\coordinate (A1) at (72:2);
\draw (A1) node {$\bullet$};
\coordinate (A2) at (144:2);
\draw (A2) node {$\bullet$};
\coordinate (A3) at (216:2);
\draw (A3) node {$\bullet$};
\coordinate (A4) at (288:2);
\draw (A4) node {$\bullet$};
\coordinate (A5) at (360:2);
\draw (A5) node {$\bullet$};
\draw (A1) -- (A2);
\draw (A1) -- (A4);
\draw (A1) -- (A5);
\draw (A2) -- (A4);
\draw (A2) -- (A3);
\end{scope}
\begin{scope}[shift={(4.8,0)}]
\draw (0,0) node {$\pm$};
\end{scope}
\begin{scope}[scale=.4,shift={(15,0)}]
\coordinate (A1) at (60:2);
\draw (A1) node {$\bullet$};
\coordinate (A2) at (120:2);
\draw (A2) node {$\bullet$};
\coordinate (A3) at (180:2);
\draw (A3) node {$\bullet$};
\coordinate (A4) at (240:2);
\draw (A4) node {$\bullet$};
\coordinate (A5) at (300:2);
\draw (A5) node {$\bullet$};
\coordinate (A6) at (360:2);
\draw (A6) node {$\bullet$};
\draw (A1) -- (A2);
\draw (A1) -- (A6);
\draw (A1) -- (A4);
\draw (A2) -- (A3);
\draw (A6) -- (A5);
\end{scope}
\begin{scope}[shift={(7.1,0)}]
\draw (0,0) node {$\pm$};
\end{scope}
\begin{scope}[scale=.4,shift={(20,0)}]
\coordinate (A1) at (72:2);
\draw (A1) node {$\bullet$};
\coordinate (A2) at (144:2);
\draw (A2) node {$\bullet$};
\coordinate (A3) at (216:2);
\draw (A3) node {$\bullet$};
\coordinate (A4) at (288:2);
\draw (A4) node {$\bullet$};
\coordinate (A5) at (360:2);
\draw (A5) node {$\bullet$};
\draw (A1) -- (A2);
\draw (A1) -- (A5);
\draw (A2) -- (A3);
\draw (A5) -- (A4);
\draw (A3) -- (A4);
\end{scope}
\begin{scope}[shift={(9.2,0)}]
\draw (0,0) node {$\pm$};
\end{scope}
\begin{scope}[scale=.4,shift={(26,0)}]
\coordinate (A1) at (51.999074699:2);
\draw (A1) node {$\bullet$};
\coordinate (A2) at (102.8571428571:2);
\draw (A2) node {$\bullet$};
\coordinate (A3) at (154.9635917385:2);
\draw (A3) node {$\bullet$};
\coordinate (A4) at (205.7142857143:2);
\draw (A4) node {$\bullet$};
\coordinate (A5) at (257.1428571429:2);
\draw (A5) node {$\bullet$};
\coordinate (A6) at (308.5714285714:2);
\draw (A6) node {$\bullet$};
\coordinate (A7) at (360:2);
\draw (A7) node {$\bullet$};
\draw (A1) -- (A2);
\draw (A1) -- (A7);
\draw (A2) -- (A3);
\draw (A7) -- (A6);
\draw (A4) -- (A5);
\end{scope}
\end{tikzpicture}
\end{center}
We recall that the coefficients in the formula above are defined up to a sign, that can be determined according to the choice of the representative of each isomorphism class of graphs as subgraph of the full graph as in Section \ref{sect:combHGammanQ}. In the following formula we explicitly choose the representatives  providing an ordering of the vertices of the graphs by means of the labeling and we compute the sign of the coefficients accordingly.

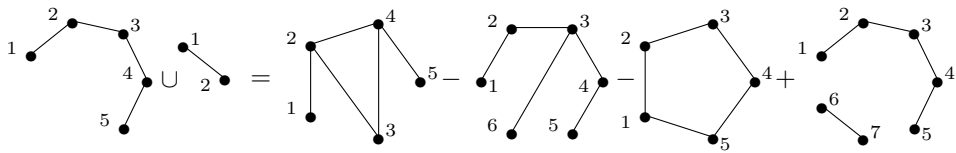
\begin{tikzpicture}
\begin{scope}[scale=.4]
\coordinate (A1) at (51.999074699:2);
\draw (A1) node {$\bune3$};
\coordinate (A2) at (102.8571428571:2);
\draw (A2) node {$\buno2$};
\coordinate (A3) at (154.9635917385:2);
\draw (A3) node {$\buno1$};
\coordinate (A4) at (205.7142857143:2);
\coordinate (A5) at (257.1428571429:2);
\coordinate (A6) at (308.5714285714:2);
\draw (A6) node {$\buno5$};
\coordinate (A7) at (360:2);
\draw (A7) node {$\buno4$};
\draw (A1) -- (A2);
\draw (A1) -- (A7);
\draw (A2) -- (A3);
\draw (A7) -- (A6);
\end{scope}

\begin{scope}[shift={(1.1,0)}]
\draw (0,0) node {$\cup$};
\end{scope}
\begin{scope}[scale=.4,shift={(5,2)}]
\coordinate (A1) at (51.999074699:2);
\coordinate (A2) at (102.8571428571:2);
\coordinate (A3) at (154.9635917385:2);
\coordinate (A4) at (205.7142857143:2);
\draw (A4) node {$\bune1$};
\coordinate (A5) at (257.1428571429:2);
\draw (A5) node {$\buso2$};
\coordinate (A6) at (308.5714285714:2);
\coordinate (A7) at (360:2);
\draw (A4) -- (A5);
\end{scope}
\begin{scope}[shift={(2.3,0)}]
\draw (0,0) node {$= $};
\end{scope}
\begin{scope}[scale=.4,shift={(9,0)}]
\coordinate (A1) at (72:2);
\draw (A1) node {$\bune4$};
\coordinate (A2) at (144:2);
\draw (A2) node {$\buno2$};
\coordinate (A3) at (216:2);
\draw (A3) node {$\buno1$};
\coordinate (A4) at (288:2);
\draw (A4) node {$\bune3$};
\coordinate (A5) at (360:2);
\draw (A5) node {$\bune5$};
\draw (A1) -- (A2);
\draw (A1) -- (A4);
\draw (A1) -- (A5);
\draw (A2) -- (A4);
\draw (A2) -- (A3);
\end{scope}
	\begin{scope}[shift={(4.8,0)}]
	\draw (0,0) node {$-$};
	\end{scope}
	\begin{scope}[scale=.4,shift={(15,0)}]
	\coordinate (A1) at (60:2);
	\draw (A1) node {$\bune3$};
	\coordinate (A2) at (120:2);
	\draw (A2) node {$\buno2$};
	\coordinate (A3) at (180:2);
	\draw (A3) node {$\buse1$};
	\coordinate (A4) at (240:2);
	\draw (A4) node {$\buno6$};
	\coordinate (A5) at (300:2);
	\draw (A5) node {$\buno5$};
	\coordinate (A6) at (360:2);
	\draw (A6) node {$\buso4$};
	\draw (A1) -- (A2);
	\draw (A1) -- (A6);
	\draw (A1) -- (A4);
	\draw (A2) -- (A3);
	\draw (A6) -- (A5);
	\end{scope}
	\begin{scope}[shift={(7.1,0)}]
	\draw (0,0) node {$-$};
	\end{scope}
	\begin{scope}[scale=.4,shift={(20,0)}]
	\coordinate (A1) at (72:2);
	\draw (A1) node {$\bune3$};
	\coordinate (A2) at (144:2);
	\draw (A2) node {$\buno2$};
	\coordinate (A3) at (216:2);
	\draw (A3) node {$\buso1$};
	\coordinate (A4) at (288:2);
	\draw (A4) node {$\buse5$};
	\coordinate (A5) at (360:2);
	\draw (A5) node {$\bune4$};
	\draw (A1) -- (A2);
	\draw (A1) -- (A5);
	\draw (A2) -- (A3);
	\draw (A5) -- (A4);
	\draw (A3) -- (A4);
	\end{scope}
	\begin{scope}[shift={(9.2,0)}]
	\draw (0,0) node {$+$};
	\end{scope}
	\begin{scope}[scale=.4,shift={(26,0)}]
	\coordinate (A1) at (51.999074699:2);
	\draw (A1) node {$\bune3$};
	\coordinate (A2) at (102.8571428571:2);
	\draw (A2) node {$\buno2$};
	\coordinate (A3) at (154.9635917385:2);
	\draw (A3) node {$\buno1$};
	\coordinate (A4) at (205.7142857143:2);
	\draw (A4) node {$\bune6$};
	\coordinate (A5) at (257.1428571429:2);
	\draw (A5) node {$\bune7$};
	\coordinate (A6) at (308.5714285714:2);
	\draw (A6) node {$\buse5$};
	\coordinate (A7) at (360:2);
	\draw (A7) node {$\bune4$};
	\draw (A1) -- (A2);
	\draw (A1) -- (A7);
	\draw (A2) -- (A3);
	\draw (A7) -- (A6);
	\draw (A4) -- (A5);
	\end{scope}

	\end{tikzpicture}

With skew coefficients, a basis for $H^3(\Gamma_4;\qsgn)$ is given by the following two graphs, the pink one
on 4 vertices surviving inside $H^3(\Gamma_5;\qsgn)$ and providing a basis there.

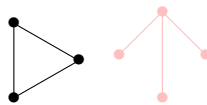
\begin{center}

\begin{tikzpicture}[scale=.19]
\coordinate (A1) at (120:3);
\draw (A1) node {$\bullet$};
\coordinate (A2) at (240:3);
\draw (A2) node {$\bullet$};
\coordinate (A3) at (360:3);
\draw (A3) node {$\bullet$};
\draw (A1) -- (A2);
\draw (A1) -- (A3);
\draw (A2) -- (A3);
\end{tikzpicture}
\begin{tikzpicture}[scale=.19,color=pink]
\coordinate (A1) at (90:3);
\draw (A1) node {$\bullet$};
\coordinate (A2) at (180:3);
\draw (A2) node {$\bullet$};
\coordinate (A3) at (270:3);
\draw (A3) node {$\bullet$};
\coordinate (A4) at (360:3);
\draw (A4) node {$\bullet$};
\draw (A1) -- (A2);
\draw (A1) -- (A3);
\draw (A1) -- (A4);
\end{tikzpicture}

\end{center}
Bases for $H^5(\Gamma_n;\qsgn)$ are provided in Table \ref{tab:H5GQeps}, with similar color conventions.

\begin{table}
$$
\begin{array}{|ll|}
\hline
\Gamma_5  &
\begin{tikzpicture}[scale=.19,color=pink]
\coordinate (A1) at (72:3);
\draw (A1) node {$\bullet$};
\coordinate (A2) at (144:3);
\draw (A2) node {$\bullet$};
\coordinate (A3) at (216:3);
\draw (A3) node {$\bullet$};
\coordinate (A4) at (288:3);
\draw (A4) node {$\bullet$};
\coordinate (A5) at (360:3);
\draw (A5) node {$\bullet$};
\draw (A1) -- (A2);
\draw (A1) -- (A3);
\draw (A1) -- (A4);
\draw (A1) -- (A5);
\draw (A2) -- (A3);
\end{tikzpicture}
\begin{tikzpicture}[scale=.19,color=pink]
\coordinate (A1) at (72:3);
\draw (A1) node {$\bullet$};
\coordinate (A2) at (144:3);
\draw (A2) node {$\bullet$};
\coordinate (A3) at (216:3);
\draw (A3) node {$\bullet$};
\coordinate (A4) at (288:3);
\draw (A4) node {$\bullet$};
\coordinate (A5) at (360:3);
\draw (A5) node {$\bullet$};
\draw (A1) -- (A3);
\draw (A1) -- (A2);
\draw (A1) -- (A5);
\draw (A2) -- (A3);
\draw (A3) -- (A4);
\end{tikzpicture}
\begin{tikzpicture}[scale=.19,color=pink]
\coordinate (A1) at (72:3);
\draw (A1) node {$\bullet$};
\coordinate (A2) at (144:3);
\draw (A2) node {$\bullet$};
\coordinate (A3) at (216:3);
\draw (A3) node {$\bullet$};
\coordinate (A4) at (288:3);
\draw (A4) node {$\bullet$};
\coordinate (A5) at (360:3);
\draw (A5) node {$\bullet$};
\draw (A1) -- (A2);
\draw (A1) -- (A3);
\draw (A1) -- (A4);
\draw (A2) -- (A3);
\draw (A4) -- (A5);
\end{tikzpicture}
\begin{tikzpicture}[scale=.19,color=pink]
\coordinate (A1) at (72:3);
\draw (A1) node {$\bullet$};
\coordinate (A2) at (144:3);
\draw (A2) node {$\bullet$};
\coordinate (A3) at (216:3);
\draw (A3) node {$\bullet$};
\coordinate (A4) at (288:3);
\draw (A4) node {$\bullet$};
\coordinate (A5) at (360:3);
\draw (A5) node {$\bullet$};
\draw (A1) -- (A2);
\draw (A2) -- (A3);
\draw (A1) -- (A5);
\draw (A3) -- (A4);
\draw (A4) -- (A5);
\end{tikzpicture}
 \\
\hline
\Gamma_6 & 
\begin{tikzpicture}[scale=.19]
\coordinate (A1) at (72:3);
\draw (A1) node {$\bullet$};
\coordinate (A2) at (144:3);
\draw (A2) node {$\bullet$};
\coordinate (A3) at (216:3);
\draw (A3) node {$\bullet$};
\coordinate (A4) at (288:3);
\draw (A4) node {$\bullet$};
\coordinate (A5) at (360:3);
\draw (A5) node {$\bullet$};
\draw (A1) -- (A2);
\draw (A1) -- (A3);
\draw (A1) -- (A4);
\draw (A1) -- (A5);
\draw (A2) -- (A3);
\end{tikzpicture}
\begin{tikzpicture}[scale=.19]
\coordinate (A1) at (72:3);
\draw (A1) node {$\bullet$};
\coordinate (A2) at (144:3);
\draw (A2) node {$\bullet$};
\coordinate (A3) at (216:3);
\draw (A3) node {$\bullet$};
\coordinate (A4) at (288:3);
\draw (A4) node {$\bullet$};
\coordinate (A5) at (360:3);
\draw (A5) node {$\bullet$};
\draw (A1) -- (A3);
\draw (A1) -- (A2);
\draw (A1) -- (A5);
\draw (A2) -- (A3);
\draw (A3) -- (A4);
\end{tikzpicture} 
\begin{tikzpicture}[scale=.19]
\coordinate (A1) at (72:3);
\draw (A1) node {$\bullet$};
\coordinate (A2) at (144:3);
\draw (A2) node {$\bullet$};
\coordinate (A3) at (216:3);
\draw (A3) node {$\bullet$};
\coordinate (A4) at (288:3);
\draw (A4) node {$\bullet$};
\coordinate (A5) at (360:3);
\draw (A5) node {$\bullet$};
\draw (A1) -- (A2);
\draw (A1) -- (A3);
\draw (A1) -- (A4);
\draw (A2) -- (A3);
\draw (A4) -- (A5);
\end{tikzpicture}
\begin{tikzpicture}[scale=.19]
\coordinate (A1) at (72:3);
\draw (A1) node {$\bullet$};
\coordinate (A2) at (144:3);
\draw (A2) node {$\bullet$};
\coordinate (A3) at (216:3);
\draw (A3) node {$\bullet$};
\coordinate (A4) at (288:3);
\draw (A4) node {$\bullet$};
\coordinate (A5) at (360:3);
\draw (A5) node {$\bullet$};
\draw (A1) -- (A2);
\draw (A1) -- (A5);
\draw (A2) -- (A3);
\draw (A4) -- (A5);
\draw (A4) -- (A3);
\end{tikzpicture}
\begin{tikzpicture}[scale=.19,color=pink]
\coordinate (A1) at (60:3);
\draw (A1) node {$\bullet$};
\coordinate (A2) at (120:3);
\draw (A2) node {$\bullet$};
\coordinate (A3) at (180:3);
\draw (A3) node {$\bullet$};
\coordinate (A4) at (240:3);
\draw (A4) node {$\bullet$};
\coordinate (A5) at (300:3);
\draw (A5) node {$\bullet$};
\coordinate (A6) at (360:3);
\draw (A6) node {$\bullet$};
\draw (A1) -- (A2);
\draw (A1) -- (A3);
\draw (A1) -- (A4);
\draw (A1) -- (A5);
\draw (A1) -- (A6);
\end{tikzpicture}
\begin{tikzpicture}[scale=.19,color=pink]
\coordinate (A1) at (60:3);
\draw (A1) node {$\bullet$};
\coordinate (A2) at (120:3);
\draw (A2) node {$\bullet$};
\coordinate (A3) at (180:3);
\draw (A3) node {$\bullet$};
\coordinate (A4) at (240:3);
\draw (A4) node {$\bullet$};
\coordinate (A5) at (300:3);
\draw (A5) node {$\bullet$};
\coordinate (A6) at (360:3);
\draw (A6) node {$\bullet$};
\draw (A1) -- (A2);
\draw (A2) -- (A3);
\draw (A2) -- (A4);
\draw (A2) -- (A5);
\draw (A1) -- (A6);
\end{tikzpicture}
\begin{tikzpicture}[scale=.19,color=pink]
\coordinate (A1) at (60:3);
\draw (A1) node {$\bullet$};
\coordinate (A2) at (120:3);
\draw (A2) node {$\bullet$};
\coordinate (A3) at (180:3);
\draw (A3) node {$\bullet$};
\coordinate (A4) at (240:3);
\draw (A4) node {$\bullet$};
\coordinate (A5) at (300:3);
\draw (A5) node {$\bullet$};
\coordinate (A6) at (360:3);
\draw (A6) node {$\bullet$};
\draw (A1) -- (A2);
\draw (A2) -- (A3);
\draw (A2) -- (A5);
\draw (A1) -- (A6);
\draw (A3) -- (A4);
\end{tikzpicture}
\begin{tikzpicture}[scale=.19,color=pink]
\coordinate (A1) at (60:3);
\draw (A1) node {$\bullet$};
\coordinate (A2) at (120:3);
\draw (A2) node {$\bullet$};
\coordinate (A3) at (180:3);
\draw (A3) node {$\bullet$};
\coordinate (A4) at (240:3);
\draw (A4) node {$\bullet$};
\coordinate (A5) at (300:3);
\draw (A5) node {$\bullet$};
\coordinate (A6) at (360:3);
\draw (A6) node {$\bullet$};
\draw (A1) -- (A2);
\draw (A2) -- (A3);
\draw (A2) -- (A4);
\draw (A1) -- (A5);
\draw (A5) -- (A6);
\end{tikzpicture}
\begin{tikzpicture}[scale=.19,color=pink]
\coordinate (A1) at (60:3);
\draw (A1) node {$\bullet$};
\coordinate (A2) at (120:3);
\draw (A2) node {$\bullet$};
\coordinate (A3) at (180:3);
\draw (A3) node {$\bullet$};
\coordinate (A4) at (240:3);
\draw (A4) node {$\bullet$};
\coordinate (A5) at (300:3);
\draw (A5) node {$\bullet$};
\coordinate (A6) at (360:3);
\draw (A6) node {$\bullet$};
\draw (A1) -- (A2);
\draw (A1) -- (A3);
\draw (A2) -- (A3);
\draw (A4) -- (A5);
\draw (A4) -- (A6);
\end{tikzpicture}
\\
\hline
\Gamma_7&
\begin{tikzpicture}[scale=.19]
\coordinate (A1) at (60:3);
\draw (A1) node {$\bullet$};
\coordinate (A2) at (120:3);
\draw (A2) node {$\bullet$};
\coordinate (A3) at (180:3);
\draw (A3) node {$\bullet$};
\coordinate (A4) at (240:3);
\draw (A4) node {$\bullet$};
\coordinate (A5) at (300:3);
\draw (A5) node {$\bullet$};
\coordinate (A6) at (360:3);
\draw (A6) node {$\bullet$};
\draw (A1) -- (A2);
\draw (A1) -- (A3);
\draw (A1) -- (A4);
\draw (A1) -- (A5);
\draw (A1) -- (A6);
\end{tikzpicture}
\begin{tikzpicture}[scale=.19]
\coordinate (A1) at (60:3);
\draw (A1) node {$\bullet$};
\coordinate (A2) at (120:3);
\draw (A2) node {$\bullet$};
\coordinate (A3) at (180:3);
\draw (A3) node {$\bullet$};
\coordinate (A4) at (240:3);
\draw (A4) node {$\bullet$};
\coordinate (A5) at (300:3);
\draw (A5) node {$\bullet$};
\coordinate (A6) at (360:3);
\draw (A6) node {$\bullet$};
\draw (A1) -- (A2);
\draw (A2) -- (A3);
\draw (A2) -- (A4);
\draw (A2) -- (A5);
\draw (A1) -- (A6);
\end{tikzpicture}
\begin{tikzpicture}[scale=.19]
\coordinate (A1) at (60:3);
\draw (A1) node {$\bullet$};
\coordinate (A2) at (120:3);
\draw (A2) node {$\bullet$};
\coordinate (A3) at (180:3);
\draw (A3) node {$\bullet$};
\coordinate (A4) at (240:3);
\draw (A4) node {$\bullet$};
\coordinate (A5) at (300:3);
\draw (A5) node {$\bullet$};
\coordinate (A6) at (360:3);
\draw (A6) node {$\bullet$};
\draw (A1) -- (A2);
\draw (A2) -- (A3);
\draw (A2) -- (A5);
\draw (A1) -- (A6);
\draw (A3) -- (A4);
\end{tikzpicture}
\begin{tikzpicture}[scale=.19]
\coordinate (A1) at (60:3);
\draw (A1) node {$\bullet$};
\coordinate (A2) at (120:3);
\draw (A2) node {$\bullet$};
\coordinate (A3) at (180:3);
\draw (A3) node {$\bullet$};
\coordinate (A4) at (240:3);
\draw (A4) node {$\bullet$};
\coordinate (A5) at (300:3);
\draw (A5) node {$\bullet$};
\coordinate (A6) at (360:3);
\draw (A6) node {$\bullet$};
\draw (A1) -- (A2);
\draw (A2) -- (A3);
\draw (A2) -- (A4);
\draw (A1) -- (A6);
\draw (A5) -- (A6);
\end{tikzpicture} 
\begin{tikzpicture}[scale=.19]
\coordinate (A1) at (60:3);
\draw (A1) node {$\bullet$};
\coordinate (A2) at (120:3);
\draw (A2) node {$\bullet$};
\coordinate (A3) at (180:3);
\draw (A3) node {$\bullet$};
\coordinate (A4) at (240:3);
\draw (A4) node {$\bullet$};
\coordinate (A5) at (300:3);
\draw (A5) node {$\bullet$};
\coordinate (A6) at (360:3);
\draw (A6) node {$\bullet$};
\draw (A1) -- (A2);
\draw (A1) -- (A3);
\draw (A2) -- (A3);
\draw (A4) -- (A5);
\draw (A4) -- (A6);
\end{tikzpicture}
\begin{tikzpicture}[scale=.19,color=pink]
\coordinate (A1) at (51.999074699:3);
\draw (A1) node {$\bullet$};
\coordinate (A2) at (102.8571428571:3);
\draw (A2) node {$\bullet$};
\coordinate (A3) at (154.9635917385:3);
\draw (A3) node {$\bullet$};
\coordinate (A4) at (205.7142857143:3);
\draw (A4) node {$\bullet$};
\coordinate (A5) at (257.1428571429:3);
\draw (A5) node {$\bullet$};
\coordinate (A6) at (308.5714285714:3);
\draw (A6) node {$\bullet$};
\coordinate (A7) at (360:3);
\draw (A7) node {$\bullet$};
\draw (A1) -- (A2);
\draw (A1) -- (A3);
\draw (A1) -- (A4);
\draw (A5) -- (A6);
\draw (A5) -- (A7);
\end{tikzpicture}
 \\
\hline
\end{array}
$$
\caption{Bases for $H^5(\Gamma_n;\qsgn)$}
\label{tab:H5GQeps}
\end{table}

\section{Generalized braid groups}
\label{sect:generalbraidgroups}
\subsection{General setting}
Let $W < \GL_{\ell}(\C)$ be a complex reflection group, $\Ar$ the corresponding hyperplane arrangement. The pure braid group is $P = \pi_1(X)$ with $X = \C^{\ell} \setminus \bigcup \Ar$. The Hurewicz identification $P^{ab} \simeq H_1(X;\Z)$ induces a $W$-equivariant isomorphism
$P^{ab} \simeq \Z \Ar$, where $\Z \Ar$
denotes the free $\Z$-module on $\Ar$ endowed
with the permutation $\Z W$-module structure arising from the natural action of $W$ on $\Ar$.
Let $B = \pi_1(X/W)$.
\begin{definition}
We consider the \emph{quasi-abelianized complex braid group} $\Gamma = B/[P,P]$ associated to the complex reflection group $W$.
\end{definition} 
If $W$ is not irreducible, $\Gamma$ will be the direct product of the quasi-abelianized braid group
associated to its irreducible factors. Therefore, we can
assume without loss of generality in the sequel that $W$ is irreducible.

From the previous considerations we get a short sequence $1 \to \Z \Ar \to \Gamma \to W \to 1$,
and from the Lyndon-Hochschild-Serre spectral sequence one gets the following straightforward generalization of Proposition \ref{prop:main}.
\begin{proposition} For any $\Q W$-module $M$ the following isomorphisms hold:
	$$H^{\bullet}(\Gamma;\Q) = H^{\bullet}(\Z \Ar;\Q)^W \simeq \Lambda^{\bullet}(\Q \Ar)^W.$$
\end{proposition}
  This provides a way to compute explicitely the Poincar\'e polynomial for $\Gamma$ in each
of the exceptional cases. We tabulate them in Tables \ref{tab:PoinExc1} to \ref{tab:PoinExc5}.
When there are dots inside the description of the Poincar\'e polynomial, this
implies this is a reciprocal polynomial, so we need to provide only half of the coefficients.
The reason why this reciprocity (often) happens is explained in the next section.

Another general fact is the following.

\begin{proposition} The Poincar\'e polynomial of $\Gamma$ admits $(1+t)^{|\Ar/W|}$
as a factor.
\end{proposition}
\begin{proof}
Let $\Ar= \Ar_1\sqcup \dots \sqcup \Ar_r$ be the decomposition of $\Ar$
in $W$-orbits, with $r = |\Ar/W|$. For each $i \in [1,r]$ the sum of the hyperplanes inside $\Ar_i$ spans
a $W$-invariant submodule 
$\Z \into \Z\Ar_i$. Considering the short exact sequence $0 \to \Z \to \Z \Ar_i \to
\Z \Ar_i/\Z \to 0$, by injectivity of $\Q$ as a $\Z$-module we can apply the functor
$H^1(\bullet;\Q) = \Hom(\bullet;\Q)$ to it and get a short exact sequence of $\Q W$-modules
$0 \to H^1(\Z \Ar_i/\Z;\Q) \to H^1(\Z \Ar_i;\Q) \to H^1(\Z;\Q) \to 0$.
By semisimplicity of $\Q W$ this provides a direct sum decomposition as $\Q W$-modules
of the form
$H^1(\Z \Ar_i;\Q) \simeq  M_i \oplus H^1(\Z;\Q)$ for some $\Q W$-module $M_i$.
Then, by K\"unneth formula $H^{\bullet}(\Z \Ar;\Q)$ is isomorphic to
$H^{\bullet}(\Z ;\Q)^{\otimes r} \otimes \Lambda^{\bullet}(M_1) \otimes \dots \otimes \Lambda^{\bullet}(M_r)$  
whence
$$
H^{\bullet}(\Gamma;\Q) \simeq H^{\bullet}(\Z ;\Q)^{\otimes r} \otimes \left(
\Lambda^{\bullet}(M_1) \otimes \dots \otimes \Lambda^{\bullet}(M_r) \right)^W
$$
which proves the claim, as $1+t$ is the Poincar\'e polynomial of $H^{\bullet}(\Z;\Q)$.
\end{proof}

As in the case of $W = \mathfrak{S}_n$, it is actually quite often the case that the Poincar\'e polynomial is divisible by a higher power of $1+t$ but, as the quotient polynomial has then negative
coefficients, we do not expect any topological reason for this.

\subsection{The $W$-module $H^N(\Z \Ar;\Q)$}

We set $N = |\Ar|$.
We have $H^N(\Gamma;\Q) \simeq H^N(\Z \Ar;\Q)^W \simeq \Lambda^N(\Q \Ar)^{W}$,
and $H^N(\Z \Ar;\Q) \simeq \Lambda^N(\Q \Ar)$ as a $W$-module. The $W$-module structure
on the 1-dimensional vector space $\Lambda^N(\Q \Ar)$ is given by the determinant of the
permutation representation $\Q \Ar$, that is the sign of the permutation action of $W$ on $\Ar$.
Since this abelian representation factorizes through $W^{ab}$, it is sufficient to
determine it on a collection of representatives of the conjugacy classes of reflections.

\begin{table}
$$
\begin{array}{cl}
\hline
G_{4} & 1+t^{1}+t^{3}+t^{4}\\ 
\hline
G_{5} & 1+2t^{1}+2t^{2}+6t^{3}+10t^{4}+6t^{5}+2t^{6}+2t^{7}+t^{8}\\ 
\hline
G_{6} & 1+2t^{1}+3t^{2}+10t^{3}+20t^{4}+24t^{5}+20t^{6}+10t^{7}+3t^{8}+2t^{9}+t^{10}\\ 
\hline
G_{7} & 1+3t^{1}+7t^{2}+30t^{3}+91t^{4}+177t^{5}+253t^{6}+284t^{7}+253t^{8}+\dots+t^{14}\\ 
\hline
G_{8} & 1+t^{1}+t^{3}+t^{4}\\ 
\hline
G_{9} & 1+2t^{1}+4t^{2}+30t^{3}+134t^{4}+376t^{5}+767t^{6}+1284t^{7}+1830t^{8}+2088t^{9}+1830t^{10} \\ 
 & +\dots+t^{18}\\
\hline
G_{10} & 1+2t^{1}+2t^{2}+15t^{3}+50t^{4}+87t^{5}+117t^{6}+142t^{7}+136t^{8}+90t^{9}+41t^{10} \\ 
 & +15t^{11}+5t^{12}+t^{13}\\ 
\hline
G_{11} &1+3t^{1}+10t^{2}+102t^{3}+647t^{4}+2790t^{5}+9543t^{6}+27262t^{7}+65211t^{8}+130587t^{9} \\ 
 & +221252t^{10}+321348t^{11}+402472t^{12}+434088t^{13}+402472t^{14}  +\dots+t^{26}\\ 
\hline
G_{12} & 1+t^{1}+t^{2}+8t^{3}+23t^{4}+38t^{5}+38t^{6}+28t^{7}+22t^{8}+13t^{9}+3t^{10}\\ 
\hline
G_{13} & 1+2t^{1}+4t^{2}+30t^{3}+134t^{4}+376t^{5}+767t^{6}+1284t^{7}+1830t^{8}+2088t^{9}+1830t^{10}
\\ & +1284t^{11}+\dots+t^{18}\\ 
\hline
G_{14} & 1+2t^{1}+5t^{2}+45t^{3}+217t^{4}+666t^{5}+1593t^{6}+3198t^{7}+5293t^{8}+7068t^{9}+7680t^{10}\\ & +6942t^{11}+5267t^{12}+3282t^{13}+1617t^{14}+630t^{15}+206t^{16}+54t^{17}+9t^{18}+t^{19}\\ 
\hline
G_{15} & 1+3t^{1}+10t^{2}+102t^{3}+647t^{4}+2790t^{5}+9543t^{6}+27262t^{7}+65211t^{8}+130587t^{9} \\ 
 & +221252t^{10}+321348t^{11}+402472t^{12}+434088t^{13}+402472t^{14} +\dots +t^{26}\\ 
\hline
G_{16} & 1+t^{1}+5t^{3}+12t^{4}+14t^{5}+14t^{6}+14t^{7}+12t^{8}+5t^{9}+t^{11}+t^{12}\\ 
\hline
\end{array}
$$
\caption{Poincar\'e polynomial of $\Gamma$ (1)}
\label{tab:PoinExc1}
\end{table}

\begin{table}
$$
\begin{array}{cl}
\hline
G_{17} & 1+2t^{1}+10t^{2}+186t^{3}+1908t^{4}+14276t^{5}+87229t^{6}+449072t^{7}+1968096t^{8}+7434074t^{9}\\ &  +24521396t^{10}+71334960t^{11}+184308107t^{12}+425331568t^{13}+880994128t^{14}\\ & +1644502390t^{15}+2775174184t^{16}+4244428460t^{17}+5894942521t^{18}+
7446171200t^{19}\\ & +8563197684t^{20}+8971058152t^{21}+8563197684t^{22}+\dots+t^{42} \\
\hline
G_{18} & 1+2t^{1}+5t^{2}+86t^{3}+636t^{4}+3362t^{5}+14983t^{6}+56130t^{7}+175775t^{8}+467520t^{9}\\ & +1074198t^{10}+2150460t^{11}+3765292t^{12}+5789700t^{13}+7854470t^{14}+9428804t^{15}\\ & +10021408t^{16}+9428804t^{17}+\dots+t^{32}
\\
\hline
G_{19}& 1+3t^{1}+25t^{2}+622t^{3}+9412t^{4}+108079t^{5}+1023742t^{6}+8194704t^{7} \\ 
 & +56357542t^{8}+338123937t^{9}+1791954240t^{10}+8471117898t^{11} \\ 
 & +36002666511t^{12}+138471654165t^{13}+484649380275t^{14}+1550878152688t^{15} \\ 
 & +4555708512956t^{16}+12327211591453t^{17}+30818019548700t^{18} \\ 
 & +71368043007570t^{19}+153441311738490t^{20}+306882630630555t^{21} \\ 
 & +571917595672410t^{22}+994639280131500t^{23} \\ 
 & +1616288882564850t^{24}+2456759132656605t^{25}+3496157157304842t^{26} \\ 
 & +4661542827614878t^{27}+5826928615825804t^{28}+6831571545737539t^{29} \\ 
 & +7514728617905480t^{30}+7757139143486168t^{31} \\ 
 & +7514728617905480t^{32}+\dots+t^{62} \\
\hline
G_{20} & 1+t^{1}+t^{2}+21t^{3}+96t^{4}+262t^{5}+621t^{6}+1302t^{7}+2157t^{
8}+2806t^{9}+3032t^{10}+2806t^{11}\\ & +2157t^{12}+1302t^{13}+621t^{14}+262t^{15}+
96t^{16}+21t^{17}+t^{18}+t^{19}+t^{20}\\
\hline
G_{21} & 1+2t^{1}+15t^{2}+320t^{3}+3912t^{4}+35460t^{5}+264448t^{6}+1663808t^{7}+8950168t^{8} \\ 
 & +41762728t^{9}+171197058t^{10}+622541248t^{11}+2023351148t^{12}+5914410488t^{13} \\ 
 & +15630708668t^{14}+37513654704t^{15}+82061595180t^{16}+164123359044t^{17} \\ 
 & +300892007724t^{18}+506765074640t^{19}+785487036924t^{20}+1122125085488t^{21} \\ 
 & +1479163461648t^{22}+1800719645168t^{23}+2025811056898t^{24} \\ 
 & +2106844801388t^{25}+2025811056898t^{26}+\dots+t^{50}\\ 

\hline
G_{22} & 1+t^{1}+4t^{2}+64t^{3}+476t^{4}+2423t^{5}+9843t^{6}+33748t^{7}+97708t^{8}+238993t^{9} \\ 
 & +500506t^{10}+909454t^{11}+1441874t^{12}+1997499t^{13}+2423604t^{14} \\ 
 & +2583668t^{15}+2423604t^{16}+\dots+t^{30}\\ 
\hline
G_{23} & 1+t^{1}+t^{2}+5t^{3}+22t^{4}+61t^{5}+93t^{6}+96t^{7}+96t^{8}+93t^{9}+\dots+t^{15}\\ 
\hline
G_{24} & 1+t^{1}+t^{2}+6t^{3}+29t^{4}+128t^{5}+355t^{6}+694t^{7}+1168t^{8}+1739t^{9}+2142t^{10} \\ 
& +2142t^{11}+1739t^{12}+\dots+t^{21}\\ 
\hline
G_{25} &1+t^{1}+2t^{3}+3t^{4}+5t^{5}+8t^{6}+5t^{7}+3t^{8}+2t^{9}+t^{11}+t^{12}\\ 
\hline
G_{26} & 1+2t^{1}+2t^{2}+8t^{3}+31t^{4}+106t^{5}+277t^{6}+555t^{7}+951t^{8}+1389t^{9}+1670t^{10} \\ 
& +1670t^{11}+1389t^{12}+\dots+t^{21}\\ 
\hline
G_{27} &1+t^{1}+t^{2}+31t^{3}+411t^{4}+3489t^{5}+22736t^{6}+125582t^{7}+598067t^{8}+2463301t^{9} \\ 
& +8865029t^{10}+28192301t^{11}+79877727t^{12}+202800457t^{13}+463560616t^{14} \\ 
& +957963894t^{15}+1796127602t^{16}+3064066154t^{17}+4766434844t^{18} \\ 
& +6773270962t^{19}+8805088666t^{20}+10482299772t^{21}+11435425556t^{22} \\ 
& +11435425556t^{23}+10482299772t^{24}+\dots+t^{45}\\ 
\hline
\end{array}
$$
\caption{Poincar\'e polynomial of $\Gamma$ (2)}
\label{tab:PoinExc2}
\end{table}

\begin{table}
$$
\begin{array}{cl}
\hline
G_{28} &1+2t^{1}+2t^{2}+4t^{3}+14t^{4}+64t^{5}+232t^{6}+626t^{7}+1329t^{8}+2308t^{9}+3370t^{10} \\ 
 & +4240t^{11}+4668t^{12}+4442t^{13}+3520t^{14}+2264t^{15}+1208t^{16} \\ 
 & +576t^{17}+260t^{18}+98t^{19}+20t^{20}\\ 
\hline
G_{29} &1+t^{1}+4t^{3}+40t^{4}+350t^{5}+2060t^{6}+9746t^{7}+39959t^{8}+142328t^{9}+441562t^{10} \\ 
 & +1204304t^{11}+2910343t^{12}+6267601t^{13}+12085830t^{14}+20949016t^{15} \\ 
 & +32736850t^{16}+46219612t^{17}+59053222t^{18}+68369460t^{19}+71790852t^{20} \\ 
 & +68383196t^{21}+59059542t^{22}+46211024t^{23}+32728563t^{24}+20951525t^{25} \\ 
 & +12091834t^{26}+6268424t^{27}+2907626t^{28}+1203008t^{29}+442248t^{30} \\ 
 & +143010t^{31}+39957t^{32}+9558t^{33}+1998t^{34}+372t^{35}+57t^{36} \\ 
 & +5t^{37}\\ 
\hline
G_{30} & 1+t^{1}+5t^{3}+74t^{4}+771t^{5}+6872t^{6}+53477t^{7}+356220t^{8}+2055627t^{9}+10468466t^{10} \\ 
 & +47582221t^{11}+194357924t^{12}+717683124t^{13}+2409187294t^{14} \\ 
 & +7387893762t^{15}+20778746068t^{16}+53781125768t^{17}+128476950008t^{18} \\ 
 & +283999914842t^{19}+582199298729t^{20}+1108953689280t^{21}+1965874330196t^{22} \\ 
 & +3247963438620t^{23}+5007273674460t^{24}+7210476160546t^{25}+9706413965476t^{26} \\ 
 & +12222890689810t^{27}+14405546870106t^{28}+15895776460298t^{29} \\ 
 & +16425637072516t^{30}+15895776460298t^{31}+\dots+t^{60} \\
\hline
G_{31} & 
1+t^{1}+5t^{3}+49t^{4}+476t^{5}+4276t^{6}+33492t^{7}+222835t^{8}+1284373t^{9}+6541542t^{10} \\ &  +29740777t^{11}+121480400t^{12}+448546734t^{13}+1505717800t^{14}
+4617446330t^{15} \\ & +12986790192t^{16}+33613184414t^{17}+80297915692t^{18}+
177499968070t^{19} \\ & +363874927936t^{20}+693096056917t^{21}+1228670832132t^{22}+2029977118521t^{23} \\ & +3129546983443t^{24}+4506547671508t^{25}
+6066507514016t^{26}+7639306616740t^{27} \\ & +9003468218232t^{28}+9934860323578t^{29}+10266021686780t^{30}
+9934860323578t^{31} \\ & +9003468218232t^{32}+7639306616740t^{33}+6066507514016t^{34}+4506547671508t^{35}
 \\ &  +3129546983443t^{36} + \dots + t^{60} \\
\hline
G_{32} & 
1+t^{1}+2t^{3}+8t^{4}+34t^{5}+156t^{6}+757t^{7}+3074t^{8}+10633t^{9}+32728t^{10}+89486t^{11} \\ 
 & +216246t^{12}+464548t^{13}+895038t^{14}+1552859t^{15}+2427296t^{16} \\ 
 & +3424118t^{17}+4373170t^{18}+5066138t^{19}+5321718t^{20}+5066138t^{21} \\ 
 & +\dots+t^{40}\\
\hline
G_{33} & 
1+t^{1}+t^{3}+7t^{4}+57t^{5}+330t^{6}+1744t^{7}+8320t^{8}+34358t^{9}+123300t^{10}+391292t^{11}\\ & +1109046t^{12}+2817848t^{13}+6440572t^{14}+13304050t^{15}+24941988t^{16}+42557847t^{17}\\ & +66209744t^{18}+94076221t^{19}+122282437t^{20}+145581450t^{
	21}+158837210t^{22}\\ & +158837210t^{23}+145581450t^{24}+122282437t^{25}+94076221t^{26}+66209744t^{27}\\ & 
+42557847t^{28} + \dots + t^{45} \\
\hline
\end{array}
$$
\caption{Poincar\'e polynomial of $\Gamma$ (3)}
\label{tab:PoinExc3}
\end{table}

\begin{table}
$$
\begin{array}{cl}
\hline
G_{34} &
1+t^{1}+t^{3}+6t^{4}+74t^{5}+884t^{6}+13116t^{7}+192326t^{8}+2520671t^{9}+2948\
9979t^{10} \\ 
 & +311014671t^{11}+2980794710t^{12}+26139611410t^{13}+210982251075t^{14} \\ 
 & +1575326752678t^{15}+10928823475928t^{16}+70715940374002t^{17}+428224373356\
294t^{18} \\ 
 & +2434117544074941t^{19}+13022528746334731t^{20}+65732763512909595t^{21} \\ 
 & +313724551754562496t^{22}+1418580582253661206t^{23}+6088075007493453094t^{2\
4} \\ 
 & +24839346043086172191t^{25}+96491305763321611190t^{26} \\ 
 & +357375206458603492081t^{27}+1263576622816720686820t^{28}\\
 & +42700175531539862\
63806t^{29} +13806390088757090593696t^{30}\\ 
 & +42755272532737758859696t^{31}+126929715331036645718996t^{32} \\ 
 & +361557370942693474774788t^{33}+988965749931371371764614t^{34} \\ 
 & +2599567114105643084316213t^{35}+6571127982881893678259444t^{36} \\ 
 & +15983824823231350150241083t^{37}+37435800243874514436429769t^{38} \\ 
 & +84470523627176973820173870t^{39}+183723388889113441013664706t^{40} \\ 
 & +385371010840651984891616378t^{41}+779917521939460133077025458t^{42} \\ 
 & +1523559810300218919356528819t^{43}+2873987823975247126613136662t^{44} \\ 
 & +5237044479243898579045032533t^{45}+9221752235190675504468316867t^{46}\\
 & +15696599549260728754834382628t^{47} +25833986758157842496665978731t^{48} \\ 
 &+41123489125230657930680765396t^{49}+63330173252855596855076193586t^{50} \\\
 
 & +94374375827785073487122176287t^{51}+136116888213151322355118222331t^{52} \\
  & +190049994863645226542541566541t^{53}+256919437500854002357709427236t^{54} \
\\ 
 & +336330900001117489347178011796t^{55}+426419533929987401697017185243t^{56} \
\\ 
 & +523673111843844950655299053043t^{57}+622990426159058822214608054411t^{58} \
\\ 
 & +718022864047728454391354376763t^{59}+801792198186627247766309064228t^{60} \
\\ 
 & +867512870169136852322333424120t^{61}+909489299370872294431572844472t^{62} \
\\ 
 & +923925637456126412476131859172t^{63}+909489299370872294431572844472t^{64} \
\\ 
 & +\dots+t^{126} \\
\hline
\end{array}
$$
\caption{Poincar\'e polynomial of $\Gamma$ (4)}
\label{tab:PoinExc4}
\end{table}

\begin{table}
$$
\begin{array}{cl}
\hline
G_{35} &
1+t^{1}+t^{4}+11t^{5}+44t^{6}+152t^{7}+566t^{8}+1860t^{9}+5004t^{10}+11572t^{11}+23972t^{12} \\ 
 & +44543t^{13}+73478t^{14}+107582t^{15}+140873t^{16}+165803t^{17} \\ 
 & +175170t^{18}+165803t^{19}+\dots+t^{36} \\
 \hline
G_{36} &
1+t^{1}+t^{4}+13t^{5}+78t^{6}+425t^{7}+2660t^{8}+16243t^{9}+87925t^{10}+423770t^{11} \\ 
 & +1838688t^{12}+7218311t^{13}+25769943t^{14}+84136890t^{15}+252379507t^{16} \\ 
 & +697845851t^{17}+1783538069t^{18}+4224088466t^{19}+9292654887t^{20} \\ 
 & +19027776867t^{21}+36326201973t^{22}+64755423636t^{23}+107924917344t^{24} \\ 
 & +168362842542t^{25}+246070368091t^{26}+337208215816t^{27}+433551008742t^{28} \\ 
 & +523248859376t^{29}+593017264581t^{30}+631280251719t^{31}+631280251719t^{32} \\ 
 & +593017264581t^{33}+\dots+t^{63} \\
 \hline
 G_{37} &
 1+t^{1}+t^{4}+11t^{5}+57t^{6}+374t^{7}+3475t^{8}+35474t^{9}+356059t^{10}+3406667t^{11} \\ 
 & +30454784t^{12}+251776769t^{13}+1921850842t^{14}+13577988054t^{15}+89107284374t^{16} \\ 
 & +545150699683t^{17}+3119546807020t^{18}+16747159056781t^{19} \\ 
 & +84573155228165t^{20}\!+\!402728709982771t^{21}\!+\!1812277301989416t^{22}\!+\!7721874254027453t^{23} \\ 
 & +31209241157562083t^{24}+119843494614454207t^{25}+437889713047953430t^{26} \\ 
 & +1524504947000485263t^{27}+5063534269228031737t^{28} \\ 
 & +16063625853649334969t^{29}+48726331566332592976t^{30}+141463543120705887706t^{31} \\ 
 & +393445479558917203986t^{32}\!+\!1049187946405641407641t^{33}\!+\!2684686805077200281859t^{34} \\ 
 & +6596659006231560771532t^{35}+15575444872422490488161t^{36} \\ 
 & +35360469435701311971865t^{37}+77234709556712473520133t^{38} \\ 
 & +162390927793545381417225t^{39}+328841628794206613232301t^{40}\\
 & +641642202531676306651320t^{41} +1206898428559421087980171t^{42} \\ 
 &+2189257614562845036307157t^{43} +3831200825453652719183382t^{44}\\ 
 & +6470472505226135626407762t^{45} +10549683432511669927601215t^{46} \\ 
 &+16610139872539671530393175t^{47} +25261254389459174519874612t^{48} \\ 
 &+37118577878253099665912691t^{49}+52708380587001828382184630t^{50} \\ 
 & +72344836099829736528394025t^{51}+95996032517233887861254109t^{52} \\ 
 & +123164720965672023324202120t^{53}+152815487124127820970155769t^{54} \\ 
 & +183378584548819686523360133t^{55}+212850142779608827070753059t^{56} \\ 
 & +238989633997983323341287019t^{57}+259592188653137826945477851t^{58} \\ 
 & +272791791466392609895802552t^{59}+277338321324370157484233484t^{60} \\ 
 & +272791791466392609895802552t^{61}+\dots+t^{120}\\
\hline
\end{array}
$$
\caption{Poincar\'e polynomial of $\Gamma$ (5)}
\label{tab:PoinExc5}
\end{table}

A uniform description of this module, for instance involving the invariant/coinvariant
theory of complex reflection groups, is missing for now. Therefore we explore it by
direct computation in each case. Using the computer package CHEVIE (see \cite{CHEVIE}) one 
checks the following, which refers to the Shephard-Todd classification of irreducible complex
reflection groups (\cite{shephard_todd}).
\begin{proposition} \label{prop:excexc}
If $W$ is an irreducible reflection group of exceptional type, then the action of $W$ on
$H^N(\Z \Ar;\Q)$ is trivial except in the following cases
$G_8$, $G_{10}$, $G_{12}$, $G_{14}$, $G_{28} = F_4$, $G_{29}$.   
\end{proposition}

\begin{definition}
Let $\eta(w)$ denote the sign of the permutation
action of $w$ on the set of reflecting hyperplanes and $\Q_{\eta}$ the corresponding $\Q W$-module.
\end{definition}
We have $H^N(\Z \Ar;\Q) \simeq \Q_{\eta}$.

Hence we have a Poincar\'e duality result in most of the
cases, as can be observed in the tables.
\begin{proposition}
If $W$ is and irreducible reflection group of exceptional type different from those listed in Proposition \ref{prop:excexc}, then $H^r(\Gamma;\Q) \simeq 
H^{N-r}(\Gamma;\Q)$.

If $W$ is and irreducible reflection group of exceptional type listed in Proposition \ref{prop:excexc}, then $H^r(\Gamma;\Q_\eta) \simeq H^{N-r}(\Gamma;\Q)$.
\end{proposition} 
\begin{proof}
The result follows from Lemma \ref{lem:pairing} applying the same arguments of Theorem \ref{thm:poincare}, cases \ref{itm:even_untwisted} and \ref{itm:odd}.  
Here we are considering the pairings
\begin{align}
	\Lambda^r(\Q\Ar) &\otimes \Lambda^{N-r}(\Q \Ar) && \to \Lambda^N(\Q \Ar) \simeq \Q  \label{eq:pairing_general_untwist}\\
	\Lambda^r(\Q\Ar) &\otimes (\Q_\eta \otimes \Lambda^{N-r}(\Q \Ar)) &&\to \Q_\eta \otimes \Lambda^N(\Q \Ar) \simeq \Q \label{eq:pairing_general_twist}
\end{align}
where the pairing in Equation \eqref{eq:pairing_general_untwist} applies to the cases not listed in Proposition \ref{prop:excexc}, while the pairing in Equation \eqref{eq:pairing_general_twist} applies to the cases listed in Proposition \ref{prop:excexc}.
\end{proof}

We note that in the case of the exceptions of Proposition \ref{prop:excexc}, $N$ is even and $\eta$ maps every distinguished reflection of odd order to $1$ and every one of even order to $-1$.

It is also possible to develop `partial' Poincar\'e duality when the action of $W$ on $\Ar$ has several orbits.
When $W$ is an irreducible reflection group, its action on $\Ar$ has at most 3 orbits,
and in rank at least $3$ it has at most 2, so let us write $\Ar = \Ar_1 \sqcup \dots \sqcup
\Ar_s$ with $s \leq 3$. 
\begin{definition}
Let  $\eta_i : W \to \GL_1(\Q) \simeq \GL(\Lambda^{N_i} \Q \Ar_i)$ with $N_i = |\Ar_i|$ denote the sign of the permutation action on the orbit $\Ar_i$ of reflecting hyperplanes 
and let $\Q_{\eta_i}$ be the associated $\Q \W_n$-module.
\end{definition}
Hence we have $\Q_\eta = \Q_{\eta_1} \otimes \cdots \otimes \Q_{\eta_s}$. 

\begin{proposition}\label{prop:partial_poincare}
We can
write $H^r(\Z \Ar;\Q) = \Lambda^r(\Q \Ar)
= \bigoplus_{i_1+\dots + i_s = r} \Lambda^{i_1}(\Q \Ar_1) \otimes \dots \otimes \Lambda^{i_s}(\Q \Ar_s)$.
Denoting $H^{i_1,\dots,i_s} =  \Lambda^{i_1}(\Q \Ar_1) \otimes \dots \otimes \Lambda^{i_s}(\Q \Ar_s)$,
whenever $\Q_{\eta_k} = 1$ we get a $W$-equivariant isomorphism
$$
H^{i_1,\dots,i_k,\dots, i_s}  \simeq H^{i_1,\dots,N_k-i_k,\dots, i_s}
$$
and hence
$$(H^{i_1,\dots,i_k,\dots, i_s})^W \simeq (H^{i_1,\dots,N_k-i_k,\dots, i_s})^W.$$
\end{proposition}
\begin{proof} The proof follows applying the duality argument to the $k$-th factor of the tensor product:
\begin{align*}
H^{i_1,\dots,i_k,\dots, i_s} &= \Lambda^{i_1}(\Q \Ar_1) \otimes \dots \otimes\Lambda^{i_k}(\Q \Ar_k)\otimes \dots \otimes \Lambda^{i_s}(\Q \Ar_s) \\
&\simeq  \Lambda^{i_1}(\Q \Ar_1) \otimes \dots \otimes (\Lambda^{N_k - i_k}(\Q \Ar_k))^*\otimes \dots \otimes \Lambda^{i_s}(\Q \Ar_s)\\
&\simeq  \Lambda^{i_1}(\Q \Ar_1) \otimes \dots \otimes (\Lambda^{N_k - i_k}(\Q \Ar_k)^*)\otimes \dots \otimes \Lambda^{i_s}(\Q \Ar_s)\\
&\simeq  \Lambda^{i_1}(\Q \Ar_1) \otimes \dots \otimes \Lambda^{N_k - i_k}(\Q \Ar_k)\otimes \dots \otimes \Lambda^{i_s}(\Q \Ar_s)\\
&\simeq  H^{i_1,\dots,N_k-i_k,\dots, i_s}. \qedhere \end{align*} 
\end{proof}

These operations will have a combinatorial interpretation for the groups $G(de,e,n)$.

\begin{table}
$$
\begin{array}{cl}
 \hline
G_8 & t^2+t^{3}+t^{5}+t^6 \\ 
 \hline
G_{10} & t^{1}\!+5t^{2}\!+15t^{3}\!+41t^{4}\!+90t^{5}\!+\!136t^{6}\!+\!142t^{7}\!+\!117t^{8}\!+87t^{9}\!+
50t^{10}\!+15t^{11}\!+2t^{12}\!+2t^{13}\! + t^{14} \\ 
\hline
G_{12} & 3t^{2}+13t^{3}+22t^{4}+28t^{5}+38t^{6}+38t^{7}+23t^{8}+8t^{9}+t^{10}+t^{
11} + t^{12} \\ 
\hline
G_{14} & t^{1}+9t^{2}+54t^{3}+206t^{4}+630t^{5}+1617t^{6}+3282t^{7}+5267t^{8}+
6942t^{9}+7680t^{10}\\ & +7068t^{11}+5293t^{12}+3198t^{13}+1593t^{14}+666t^{15}+
217t^{16}+45t^{17}+5t^{18}+2t^{19} + t^{20} \\ 
\hline
G_{28} & 20t^{4}+98t^{5}+260t^{6}+576t^{7}+1208t^{8}+2264t^{9}+3520t^{10}+
4442t^{11}+4668t^{12} +4240t^{13}\\ &+3370t^{14}+2308t^{15}+1329t^{16}+626t^{17}+
232t^{18}+64t^{19}+14t^{20}+4t^{21}+2t^{22}+ 2t^{23} + t^{24} \\ 
\hline
G_{29} & 5t^{3}+57t^{4}+372t^{5}+1998t^{6}+9558t^{7}+39957t^{8}+143010t^{9}+442248t^{
10}+1203008t^{11}\\ & +2907626t^{12}+6268424t^{13}+12091834t^{14}+20951525t^{15}+
32728563t^{16}+46211024t^{17}\\ & +59059542t^{18}+68383196t^{19}+71790852t^{20}+
68369460t^{21}+59053222t^{22}+46219612t^{23}\\ & +32736850t^{24}+20949016t^{25}+
12085830t^{26}+6267601t^{27}+2910343t^{28}+1204304t^{29}\\ & +441562t^{30}+
142328t^{31}+39959t^{32}+9746t^{33}+2060t^{34}+350t^{35}+40t^{36}+4t^{37}+t^{
39}\\ & + t^{40}\\ 
\hline
\end{array}
$$
\caption{Poincar\'e polynomials for $H^{\bullet}(\Gamma;\Q_{\eta})$}
\end{table}

\subsection{Groups $G(e,e,n)$}
We now focus on the case of the complex reflection group $W$ of type $(e,e,n)$ and the corresponding quasi-abelianized complex braid group $\Gamma(e,e,n)$. 
\begin{proposition} \label{prop:sign_eer}
Let us consider a reflection around some hyperplane $z_i = \zeta z_j$ for $i \neq j$, $\zeta \in \mu_e$.
The sign of its action on
the hyperplanes is $(-1)^{(n-2)e + (e-2)/2} = (-1)^{(e-2)/2}$ if $e$ is even,
and $(-1)^{(n-2)e + (e-1)/2} = (-1)^{n + (e-1)/2}$ if $e$ is odd.
\end{proposition}
\begin{proof}
The non-fixed hyperplanes are the ones of the form $z_a = \xi z_b$ for $a \in \{ i,j \}$ and $b \not\in \{ i,j \}$,
and there are $2(n-2)e$ of them, and all the hyperplanes $z_i = \xi z_j$, except for the hyperplane $z_i = \zeta z_j$ itself,
and possibly the hyperplane $z_i = -\zeta z_j$ if $e$ is even. Therefore the claim follows.
\end{proof}

In this case, the cohomology classes can be described as follows.
\begin{definition}
We denote by $\full_n(e)$ the complete $e$-multigraph on the
$n$ vertices $\{ 1,\dots , n \}$
with $e$ edges between each pair of vertices, labelled by the set $\mu_e$ of roots of $1$. 
We fix an ordering of $\mu_e$. 
We call $e$-multigraph on $n$ vertices any subgraph of this labelled multigraph.
\end{definition} 
To an edge between $i$ and $j$ with $i < j$ and label
$\zeta$, we associate the 1-form $\omega_{i,j}^{\zeta} = \dlog (z_i - \zeta z_j)$, and to
a $e$-multigraph we associate the wedge product of the 1-forms associated to its edges, ordered lexicographically
from the tuple $(i,j,\zeta)$.

The $W$-action on such forms has a natural combinatorial translation on the collection of $e$-multigraphs. In particular,
the reflection $(z_1,z_2,\dots, z_n) \mapsto (\xi z_2,\xi^{-1}z_1,z_3,\dots,z_n)$ acts as follows.
\begin{center}

\begin{tikzpicture}
\begin{scope}
\draw[ultra thick] (0,0) -- (2,0);
\draw[ultra thick,red] (2,0) -- (4,0);
\draw[ultra thick,blue] (0,0) -- (-2,1);
\draw[ultra thick,green] (0,0) -- (-2,-1);
\draw[ultra thick] (0,0) -- (2,0);
\draw (0,0) node {$\bullet$};
\draw (2,0) node {$\bullet$};
\draw (-2,1) node {$\bullet$};
\draw (-2,-1) node {$\bullet$};
\draw (4,0) node {$\bullet$};
\draw (1,-0.25) node {$\zeta$};
\draw (0,-0.2) node {$\mathbf{1}$};
\draw (2,-0.2) node {$\mathbf{2}$};
\draw (-2,0.8) node {$\mathbf{i}$};
\draw (-2,-1.25) node {$\mathbf{j}$};
\draw (4,-0.2) node {$\mathbf{k}$};
\draw (3,-0.2) node {$\zeta_k$};
\draw (-1,0.3) node {$\zeta_i$};
\draw (-1,-0.75) node {$\zeta_j$};
\end{scope}
\draw (5,0) node {$\mapsto$};
\begin{scope}[shift={(8,0)}]
\draw[ultra thick] (0,0) -- (2,0);
\draw[ultra thick,red] (0,0) -- (-2,0);
\draw[ultra thick,blue] (2,0) -- (4,1);
\draw[ultra thick,green] (2,0) -- (4,-1);
\draw[ultra thick] (0,0) -- (2,0);
\draw (0,0) node {$\bullet$};
\draw (2,0) node {$\bullet$};
\draw (4,1) node {$\bullet$};
\draw (4,-1) node {$\bullet$};
\draw (-2,0) node {$\bullet$};
\draw (1,-0.25) node {$\zeta^{-1}\xi^2$};
\draw (0,-0.2) node {$\mathbf{1}$};
\draw (2,-0.2) node {$\mathbf{2}$};
\draw (4,0.8) node {$\mathbf{i}$};
\draw (4,-1.25) node {$\mathbf{j}$};
\draw (-2,-0.2) node {$\mathbf{k}$};
\draw (-1,-0.2) node {$\xi \zeta_k$};
\draw (3.2,0.3) node {$\xi^{-1} \zeta_i$};
\draw (3.2,-0.8) node {$\xi^{-1} \zeta_j$};
\end{scope}
\end{tikzpicture}

\end{center}

\begin{definition}
	Let $\Stab_W(\Delta)$ denote the subgroup of $W$ preserving the multigraph $\Delta$.
\end{definition} 
\begin{definition}
A $e$-multigraph $\Delta$ is called \emph{invariant} if every 
element of $\Stab_W(\Delta)$ induces an even permutation of the edges.
\end{definition}
As in section \ref{sect:combHGammanQ}, Theorem \ref{thm:generators} and \ref{thm:skewgenerators}, we have the following result. The proof uses the same argument.
\begin{theorem}
A basis for the cohomology group $H^\bullet(\Gamma(e,e,n);\Q)$ is naturally indexed by the orbits
under the action of $W$ of the collection of invariant $e$-multigraphs.
\end{theorem}  
A similar formula for the cup-product holds as well.

A similar Poincar\'e duality holds.
\begin{proposition}
Let $\Gamma$ be the quasi-abelianized complex braid group of type $(e,e,n)$.
In the following cases:
\begin{enumerate}
	\item when 	$e$ even and $(e-2)/2$ is even,
	\item when $e$ odd and $n + (e-1)/2$ is even
\end{enumerate} 
we have an isomorphism
$H^r(\Gamma;\Q) \simeq H^{N-r}(\Gamma;\Q).$
\end{proposition}
\begin{proof}
Because of Proposition \ref{prop:sign_eer} the cases above are exactly those when $\Q_{\eta}=1$, so we can apply Lemma \ref{lem:pairing}. 
\end{proof}
Similar to the classical case, also here the Poincaré duality is combinatorially described by taking the complement $e$-multigraph.

The hyperplane arrangement in this case contains only one orbit, except possibly when $n=2$, that is when $W$ is a dihedral group. In this case the full multigraph can be identified with $\mu_e(\C)$. There are two orbits exactly when  
$e$ is even, and these are identified with the classes modulo $\mu_2(\C) \subset \mu_e(\C)$. We have $\Q_{\eta_1} = \Q_{\eta_2} = \Q$ if and only if $e \equiv 2 \mod 4$, and the partial Poincar\'e duality is obtained by taking the complement in each of the classes.

We tabulate in Table \ref{tab:BettiDn} the first Betti numbers in Coxeter type $D_n$ corresponding to the case $e = 2$. 
As an example,
the 2-multigraphs appearing for $n=4$ and $r =4$ are the ones pictured below, where a black edges means a $+1$ label, and a blue one a $-1$ label. 
Both are mapped to $0$ in the 1-dimensional cohomology group for $n = 5$, a basis of which being  
the linear graph on $5$ vertices with labels 1, which is inherited
from $\Gamma_n$. We shall prove later on that this class stabilizes.

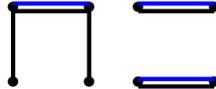
\begin{center}

\begin{tikzpicture}[scale=.5]
\draw (-1,1) node {$\bullet$}; 
\draw (1,1) node {$\bullet$}; 
\draw (-1,-1) node {$\bullet$}; 
\draw (1,-1) node {$\bullet$};
\draw[blue,ultra thick] (-1,1+.1) -- (1,1+.1); 
\draw[ultra thick] (-1,1-.1) -- (1,1-.1); 
\draw[ultra thick] (-1,-1) -- (-1,1);
\draw[ultra thick] (1,-1) -- (1,1);
\end{tikzpicture}
\ 
\begin{tikzpicture}[scale=.5]
\draw (-1,1) node {$\bullet$}; 
\draw (1,1) node {$\bullet$}; 
\draw (-1,-1) node {$\bullet$}; 
\draw (1,-1) node {$\bullet$};
\draw[blue,ultra thick] (-1,1+.1) -- (1,1+.1); 
\draw[ultra thick] (-1,1-.1) -- (1,1-.1); 
\draw[blue,ultra thick] (-1,-1+.1) -- (1,-1+.1); 
\draw[ultra thick] (-1,-1-.1) -- (1,-1-.1); 
\end{tikzpicture}

\end{center}
\begin{table}
$$
\begin{array}{|c||c|c|c|c|c|c|c|c|c|c|c|c|c|c|c|}
\hline
n & 0 & 1 & 2 & 3 & 4 & 5 & 6 & 7 & 8 & 9 & 10 & 11 & 12 & 13  \\
\hline
\hline
2 & 1 & 2 & 1 &&&&&&&&&&\\ 
 \hline 
3 & 1 & 1 & 0 & 0 & 0 & 1 & 1&&&&&&&\\ 
 \hline 
4 & 1 & 1 & 0 & 1 & 2 & 10 & 18 & 10 & 2 & 1 & 0 & 1 & 1 &\\ 
 \hline 
5 & 1 & 1 & 0 & 0 & 1 & 11 & 27 & 38 & 55 & 90 & 112 & 90 & 55&\dots \\ 
 \hline 
6 & 1 & 1 & 0 & 0 & 1 & 17 & 64 & 171 & 473 & 1267 & 2758 & 4834 & 7322 &\dots \\ 
 \hline 
7 & 1 & 1 & 0 & 0 & 1 & 14 & 49 & 122 & 387 & 1440 & 4741 & 13401 & 33899 &\dots \\ 
 \hline 
8 & 1 & 1 & 0 & 0 & 1 & 14 & 53 & 158 & 630 & 3030 & 13848 & 57350 & 215531 &\dots \\ 
 \hline 
\end{array}
$$
\caption{First Betti numbers for $\Gamma$ in Coxeter type $D_n$}
\label{tab:BettiDn}
\end{table}

\subsection{Groups $G(de,e,n)$, $d>1$}
We now focus on the case of the complex reflection group $W$ of type $(e,e,n)$ and the corresponding quasi-abelianized complex braid group $\Gamma(de,e,n)$. 
In this case, the cohomology classes can be described as follows. 
\begin{definition}
We denote by $\tilde{K}_n(de)$ the full multigraph with loops on the
$n$ vertices $\{ 1,\dots , n \}$
with $de$ edges between each pair of vertices, labelled by the set $\mu_{de}$ of roots of $1$, and one unlabelled loop per vertex.
We fix an ordering of $\mu_{de}$. We $de$-multigraph with loops
on $n$ vertices any subgraph of $\tilde{K}_n(de)$.
\end{definition} 
To an edge between $i$ and $j$ with $i < j$ and label
$\zeta$, we associate the 1-form $\omega_{i,j}^{\zeta} = \mathrm{d}\log (z_i - \zeta z_j)$, to a loop around the vertex $i$ we
associate the 1-form $\mathrm{d}\log (z_i)$, and to
a $e$-multigraph we associate the wedge product of the 1-forms associated to its edges and loops, ordered lexicographically
from the tuple $(i,j,\zeta)$ for the edges, loops ordered by their label, and edges considered smaller than loops.

The $W$-action on such forms has a natural combinatorial translation on the collection of $e$-multigraphs with loops. The loops are
interverted according to the permutation associated to $w \in G(de,e,n)$ under the natural morphism $G(de,e,n) \to \W_n$.
The reflection $(z_1,z_2,\dots, z_n) \mapsto (\xi z_2,\xi^{-1}z_1,z_3,\dots,z_n)$ acts as before on the edges,
and the reflections which are not conjugates of this one are conjugates of a reflection of the
form $(z_1,z_2,\dots, z_n) \mapsto (\xi z_1,z_2,\dots,z_n)$, and this one acts on the edges as follows

\begin{center}
\begin{tikzpicture}
\begin{scope}
\draw[ultra thick] (0,0) -- (2,0);
\draw[ultra thick,red] (2,0) -- (4,0);
\draw[ultra thick,blue] (0,0) -- (-2,1);
\draw[ultra thick,green] (0,0) -- (-2,-1);
\draw[ultra thick] (0,0) -- (2,0);
\draw (0,0) node {$\bullet$};
\draw (2,0) node {$\bullet$};
\draw (-2,1) node {$\bullet$};
\draw (-2,-1) node {$\bullet$};
\draw (4,0) node {$\bullet$};
\draw (1,-0.25) node {$\zeta$};
\draw (0,-0.2) node {$\mathbf{1}$};
\draw (2,-0.2) node {$\mathbf{2}$};
\draw (-2,0.8) node {$\mathbf{i}$};
\draw (-2,-1.25) node {$\mathbf{j}$};
\draw (4,-0.2) node {$\mathbf{k}$};
\draw (3,-0.2) node {$\zeta_k$};
\draw (-1,0.3) node {$\zeta_i$};
\draw (-1,-0.75) node {$\zeta_j$};
\end{scope}
\draw (5,0) node {$\mapsto$};
\begin{scope}[shift={(8,0)}]
\draw[ultra thick] (0,0) -- (2,0);
\draw[ultra thick,red] (2,0) -- (4,0);
\draw[ultra thick,blue] (0,0) -- (-2,1);
\draw[ultra thick,green] (0,0) -- (-2,-1);
\draw[ultra thick] (0,0) -- (2,0);
\draw (0,0) node {$\bullet$};
\draw (2,0) node {$\bullet$};
\draw (-2,1) node {$\bullet$};
\draw (-2,-1) node {$\bullet$};
\draw (4,0) node {$\bullet$};
\draw (1,-0.25) node {$\xi \zeta$};
\draw (0,-0.2) node {$\mathbf{1}$};
\draw (2,-0.2) node {$\mathbf{2}$};
\draw (-2,0.8) node {$\mathbf{i}$};
\draw (-2,-1.25) node {$\mathbf{j}$};
\draw (4,-0.2) node {$\mathbf{k}$};
\draw (3,-0.2) node {$\zeta_k$};
\draw (-1,0.3) node {$\xi \zeta_i$};
\draw (-1,-0.75) node {$\xi \zeta_j$};
\end{scope}
\end{tikzpicture}
\end{center}

\begin{definition}
	Let $\Stab_W(\Delta)$ denote the subgroup of $W$ preserving the multigraph with loops $\Delta$.
\end{definition} 
\begin{definition}
A $de$-multigraph with loops $\Delta$ is called invariant if every 
	element of $\Stab_W(\Delta)$ induces an even permutation of the edges.
\end{definition}
As in section \ref{sect:combHGammanQ} we have a description an additive basis of the cohomology. The proof uses the same argument.
\begin{theorem}
A basis for the cohomology group $H^\bullet(\Gamma(de,e,n);\Q)$ is naturally indexed by the orbits
under the action of $W$ of the collection of invariant $de$-multigraphs with loops.
\end{theorem}  
A similar formula for the cup-product holds as well.

For $n \geq 3$, there are two orbits of hyperplanes, the one corresponding to edges and the one corresponding to loops. We denote
$\eta_1,\eta_2$ the corresponding characters, and compute them now. 
\begin{proposition}
Let $\rho$ be the reflection 
$$(z_1,z_2,\dots, z_n) \mapsto (\xi z_2,\xi^{-1}z_1,z_3,\dots,z_n)$$ 
and let $\tau$ be the reflection
$$\tau:(z_1,z_2,\dots, z_n) \mapsto (\xi z_1,z_2,\dots,z_n).$$
The value of $\eta_1\eta_2$ on the reflection $\rho$ is given as follows: $\rho$ acts on $H^N(\Z \Ar;\Q)$ as $1$ if and only if one of the following condition is satisfied:
\begin{enumerate}
	\item $de$ is even and $(de-2)/2$ is even,
	\item $de$ is odd and $n + (de-1)/2$ is even.
\end{enumerate} 
The action of $\rho$ on loops is given by a transposition, whence the value of $\eta_2$ on $\rho$
is $-1$. 
The value of $\eta_1$ on $\rho$ is then deduced from the value of $\eta_1\eta_2$ and $\eta2$.
The reflection $\tau$ fixes the loops, hence the value of $\eta_2$ on $\tau$ is $1$. 
The reflection $\tau$ acts on edges with
cycles of length the order $o(\xi)$ of $\xi$ inside $\C^{\times}$. Since there are $(n-1)de/o(\xi)$ such cycles,
 the values of $\eta_1$ on $\tau$ is $1$ if and only if one of the following conditions is satisfied
 \begin{enumerate}
 	\item $o(\xi)$ is odd,  
 	\item both $o(\xi)$ and $(n-1)de/o(\xi)$ are even. 
 \end{enumerate}
\end{proposition}

When $\Q_{\eta_1} = \Q$, a partial duality, as described in Proposition \ref{prop:partial_poincare}, is obtained by taking the complement inside the collection of edges of the full multigraph $\tilde{K}_n(de)$. 

\begin{remark}
	Note that $\Q_{\eta} = \Q_{\eta_1}\otimes \Q_{\eta_2}$ is never equal to the trivial module in this case.
\end{remark}
\begin{table}
$$
\begin{array}{|c||c|c|c|c|c|c|c|c|c|c|c|c|c|c|c|}
\hline 
n & 0 & 1 & 2 & 3 & 4 & 5 & 6 & 7 & 8 & 9 & 10 & 11 & 12 & 13  \\
\hline
\hline
2 & 1 & 2 & 1 & 0 & 0 &&&&&&&&&\\ 
 \hline 
3 & 1 & 2 & 2 & 2 & 2 & 5 & 7 & 3 & 0 & 0&&&&\\ 
 \hline 
4 & 1 & 2 & 2 & 3 & 6 & 20 & 46 & 64 & 66 & 59 & 46 & 27 & 9& \dots \\ 
 \hline 
5 & 1 & 2 & 2 & 3 & 9 & 36 & 109 & 254 & 524 & 1017 & 1724 & 2388 & 2728& \dots \\ 
 \hline 
6 & 1 & 2 & 2 & 3 & 9 & 43 & 156 & 467 & 1383 & 4081 & 11027 & 26065 & 
53897& \dots \\ 
 \hline 
7 & 1 & 2 & 2 & 3 & 9 & 44 & 175 & 591 & 2090 & 7853 & 28545 & 95611 & 
292529& \dots \\ 
 \hline 
8 & 1 & 2 & 2 & 3 & 9 & 44 & 178 & 632 & 2425 & 10295 & 44336 & 184803 & 
735485& \dots \\ 
 \hline 
\end{array}
$$
\caption{First Betti numbers for $\Gamma$ in Coxeter type $B_n$}
\end{table}

\subsection{Stabilization}
\label{sect:stabilization}

We first consider the case of $G(d,1,n)$. The complete $d$-multigraph with loops on $n$ vertices $\tilde{K}_n(d)$ is naturally identified to a subgraph
of $\tilde{K}_{n+1}(d)$, where the vertex $n+1$ has no loops and no incident edge. This enables us to
identify any $d$ multigraph $\Delta$ with loops on $n$ vertices with a multigraph with loops on $n+1$ vertices.

\begin{proposition}
The natural maps 
$H^r(\Gamma(d,1,n+1);\Q) \to H^r(\Gamma(d,1,n);\Q)$
are surjective. Moreover $H^r(\Gamma(d,1,n);\Q)$  stabilizes
for $n \geq 2r$.
\end{proposition}
\begin{proof}
We clearly have $\Stab_{G(d,1,n)}(\Delta) = \Stab_{G(d,1,n+1)}(\Delta) \cap G(d,1,n)$ and this implies
the first part of the statement.
Moreover, since any $d$-multigraph with loops with at most $r$ edges can be realized as
a subgraph of $\tilde{K}_n(d)$ with $n \geq 2r$, in that range the cohomology groups have the same
dimension, and therefore we obtain the stability result.
\end{proof}

\begin{theorem} Let $d,e\geq 1$ be integers.
	The inclusion map $\Gamma(de,e,n) \into \Gamma(de,e,n+1)$ induces stabilization of the rational
	cohomology groups of degree $r$ for $n \geq 2r+1$ (and even $n \geq 2r$ when $e=1$).
\end{theorem}
\begin{proof}
	In general, we have the following commutative diagram of groups:
\begin{center}
\begin{tikzcd}
G(de,e,n) \arrow[r] \arrow[d, symbol=\subset]& G(de,e,n+1)\arrow[d, symbol=\subset] \\
G(de,1,n) \ar[r] \arrow[ur, hookrightarrow] & G(de,1,n+1)
\end{tikzcd}
\end{center}
where the map $\iota : G(de,1,n) \to G(de,e,n+1)$ is given by $M \mapsto \mathrm{diag}(M,m^{-1})$ where $m$ is the product of the
non-zero entries in $M$. 

It is readily checked that the groups $\iota(G(de,1,n))$ and $\Stab_{G(de,e,n+1)}(\tilde{K}_n(de))$ have the same
image inside the group of permutations of the edges and loops of $\tilde{K}_n(de)$.

 This implies for $d > 1$ that
the natural map 
$p_n : H^r(\Gamma(de,1,n);\Q) \to H^r(\Gamma(de,e,n+1);\Q)$
 is surjective for all $r$. 
 
Assume $n \geq 2r$. 
We have the following diagram of cohomology groups
$$
\xymatrix{
H^r(\Gamma(de,e,n);\Q)  & 
H^r(\Gamma(de,e,n+1);\Q)\ar[l]_{a_n} \ar@{->>}[dl]_{p_n} &
 H^r(\Gamma(de,e,n+2);\Q)\ar[l]_{a_{n+1}} \ar@{->>}[dl]_{p_{n+1}} \\
H^r(\Gamma(de,1,n);\Q)\ar[u]^{c_n}  & 
H^r(\Gamma(de,1,n+1);\Q)\ar[u]^{c_{n+1}} 
\ar[l]_{b_n}^{\simeq} & 
H^r(\Gamma(de,1,n+2);\Q)\ar[u]^{c_{n+2}}\ar[l]_{b_{n+1}}^{\simeq}  \\
}
$$

Then $a_n$ and $a_{n+1}$ are injective, as all multigraphs with $r$ edges can be realized
on $2r$ vertices. From the surjectivity of the maps $p_n$, $p_{n+1}$, this readily implies
that the vertical maps $c_n$, $c_{n+1}$ are injective. By elementary diagram chasing this
implies that $a_{n+1}$ is surjective.
Thus $a_{n+1}$ is an isomorphism, and the sequence $H^r(\Gamma(de,e,n);\Q)$ stabilizes
for $n \geq 2r+1$.

\smallskip

For $d=1$ we can define $q_n$ as the composite of $p_n$ with the obvious projection map from
$H^r(\Gamma(de,1,n);\Q)$ to its submodule $H^r_0(\Gamma(de,1,n);\Q)$ spanned
by the invariant graphs with no loops.
Then, for $n \geq 2r$, the restriction $b_n^0 : H^r_0(\Gamma(de,1,n+1);\Q)\to H^r_0(\Gamma(de,1,n);\Q)$ of $b_n$ is an isomorphism and the same argument can be applied to the diagram
$$
\xymatrix{
H^r(\Gamma(e,e,n);\Q)  & 
H^r(\Gamma(e,e,n+1);\Q)\ar[l]_{a_n} \ar@{->>}[dl]_{q_n} &
 H^r(\Gamma(e,e,n+2);\Q)\ar[l]_{a_{n+1}} \ar@{->>}[dl]_{q_{n+1}} \\
H_0^r(\Gamma(de,1,n);\Q)\ar[u]^{c_n}  & 
H_0^r(\Gamma(de,1,n+1);\Q)\ar[u]^{c_{n+1}} 
\ar[l]_{b_n^0}^{\simeq} & 
H_0^r(\Gamma(de,1,n+2);\Q)\ar[u]^{c_{n+2}}\ar[l]_{b_{n+1}^0}^{\simeq}  \\
}
$$
thus providing stability again for $n \geq 2r+1$.

This completes the proof of the Theorem.
\end{proof}

\providecommand{\bysame}{\leavevmode\hbox
	to3em{\hrulefill}\thinspace}

\bibliographystyle{abbrv}
\bibliography{biblio}

\begin{thebibliography}{10}

\bibitem{ADEMMILGRAM}
A.~Adem and R.~J. Milgram.
\newblock {\em Cohomology of finite groups.}, volume 309 of {\em Grundlehren
  Math. Wiss.}
\newblock Berlin: Springer, 2nd ed. edition, 2004.

\bibitem{ARNOLD}
V.~I. Arnol'd.
\newblock Topological invariants of algebraic functions. {II}.
\newblock {\em Funct. Anal. Appl.}, 4:91--98, 1970.

\bibitem{BECKMAR}
V.~Beck and I.~Marin.
\newblock Torsion subgroups of quasi-abelianized braid groups.
\newblock {\em J. Algebra}, 558:3--23, 2020.

\bibitem{CALMAR}
F.~Callegaro and I.~Marin.
\newblock Homology computations for complex braid groups.
\newblock {\em J. Eur. Math. Soc. (JEMS)}, 16(1):103--164, 2014.

\bibitem{CALSALV}
F.~Callegaro and M.~Salvetti.
\newblock Families of superelliptic curves, complex braid groups and
  generalized {Dehn} twists.
\newblock {\em Isr. J. Math.}, 238(2):945--1000, 2020.

\bibitem{church_farb}
T.~Church and B.~Farb.
\newblock Representation theory and homological stability.
\newblock {\em Adv. Math.}, 245:250--314, 2013.

\bibitem{GGO}
D.~L. Gon{\c{c}}alves, J.~Guaschi, and O.~Ocampo.
\newblock A quotient of the {Artin} braid groups related to crystallographic
  groups.
\newblock {\em J. Algebra}, 474:393--423, 2017.

\bibitem{THMARIN}
I.~Marin.
\newblock {\em {Repr{\'e}sentations lin{\'e}aires des tresses
  infinit{\'e}simales}}.
\newblock Theses, {Universit{\'e} Paris XI-Orsay}, 2001.

\bibitem{CRYSTMARIN}
I.~Marin.
\newblock Crystallographic groups and flat manifolds from complex reflection
  groups.
\newblock {\em Geom. Dedicata}, 182:233--247, 2016.

\bibitem{CHEVIE}
J.~Michel.
\newblock The development version of the \texttt{CHEVIE} package of
  \texttt{GAP3}.
\newblock {\em J. Algebra}, 435:308--336, 2015.

\bibitem{ORLIKTERAO}
P.~Orlik and H.~Terao.
\newblock {\em Arrangements of hyperplanes}, volume 300 of {\em Grundlehren
  Math. Wiss.}
\newblock Berlin: Springer-Verlag, 1992.

\bibitem{PANAITESAIC}
F.~Panaite and M.~D. Staic.
\newblock A quotient of the braid group related to pseudosymmetric braided
  categories.
\newblock {\em Pac. J. Math.}, 244(1):155--167, 2010.

\bibitem{RWW}
O.~Randal-Williams and N.~Wahl.
\newblock Homological stability for automorphism groups.
\newblock {\em Adv. Math.}, 318:534--626, 2017.

\bibitem{shephard_todd}
G.~C. Shephard and J.~A. Todd.
\newblock Finite unitary reflection groups.
\newblock {\em Can. J. Math.}, 6:274--304, 1954.

\bibitem{spanier}
E.~H. Spanier.
\newblock {\em Algebraic topology}.
\newblock Berlin: Springer-Verlag, 1995.

\bibitem{SYSOEVA}
I.~Sysoeva.
\newblock Dimension {{\(n\)}} representations of the braid group on {{\(n\)}}
  strings.
\newblock {\em J. Algebra}, 243(2):518--538, 2001.

\bibitem{TITS}
J.~Tits.
\newblock Normalisateurs de tores. {I}: {Groupes} de {Coxeter} etendus.
\newblock {\em J. Algebra}, 4:96--116, 1966.

\bibitem{TONGYANGMA}
D.~Tong, S.~Yang, and Z.~Ma.
\newblock A new class of representations of braid groups.
\newblock {\em Commun. Theor. Phys.}, 26(4):483--486, 1996.

\bibitem{W}
N.~Wahl.
\newblock Homological stability for mapping class groups of surfaces.
\newblock In {\em Handbook of moduli. Volume III}, pages 547--583. Somerville,
  MA: International Press; Beijing: Higher Education Press, 2013.

\end{thebibliography}
\nocite{*}

\end{document}